\documentclass[12pt]{amsart}

\usepackage{cases}
\usepackage[english]{babel}
\usepackage{times,bm,amsfonts,amsmath,amssymb,dsfont} 
\usepackage{graphicx} 
\usepackage[babel=true]{csquotes}
\usepackage{mathabx}
\usepackage{pdfsync}

\makeatletter
\def\section{\@startsection{section}{1}\z@{.9\linespacing\@plus\linespacing}%
  {.7\linespacing} {\fontsize{13}{15}\selectfont\scshape\centering}}
\def\paragraph{\@startsection{paragraph}{4}%
  \z@{0.3em}{-.5em}%
  {$\bullet$ \ \normalfont\itshape}}
\makeatother

\newtheorem{theo}{Theorem}[section]
\newtheorem{prop}[theo]{Proposition}
\newtheorem{lem}[theo]{Lemma}
\newtheorem{cor}[theo]{Corollary}

\theoremstyle{definition}

\theoremstyle{remark}
\newtheorem{rem}[theo]{Remark}
\newcommand\got[1]{{\bm{\mathfrak{#1}}}}
\makeatletter

\@addtoreset{equation}{section}
\makeatother

\usepackage{color}

\definecolor{gr}{rgb}   {0.,   0.69,   0.23 }
\definecolor{bl}{rgb}   {0.,   0.5,   1. }
\definecolor{mg}{rgb}   {0.85,  0.,    0.85}
\definecolor{yl}{rgb}   {0.8,  0.7,   0.}

\newcommand{\Bk}{\color{black}}

\newcommand{\Gr}{\color{gr}}

\newcommand{\N}{\mathbb{N}}
\newcommand{\R}{\mathbb{R}}
\newcommand{\Z}{\mathbb{Z}}

\newcommand\cF{\mathcal{F}}
\newcommand\cC{\mathcal{C}}

\newcommand\cM{\mathcal{M}}
\newcommand\cO{\mathcal{O}}

\newcommand\gN{\got{N}}

\renewcommand\gg{\got{g}}
\newcommand\gq{\got{q}}
\newcommand\gu{\got{u}}

\newcommand\gphi{\got{\phi}}

\newcommand\spec{\got{S}}
\newcommand\tP{\widetilde{P}}
\newcommand\tg{\widetilde{g}}
\newcommand\tgg{\widetilde{\got{g}}}
\newcommand\tq{\widetilde{q}}
\newcommand\tgq{\widetilde{\got{q}}}
\newcommand\tu{\widetilde{u}}
\newcommand\tgu{\widetilde{\got{u}}}

\newcommand\hgq{\widehat{\got{q}}}
\newcommand\hgg{\widehat{\got{g}}}
\newcommand\hgu{\widehat{\got{u}}}
\newcommand\hgv{\widehat{\got{v}}}

\newcommand{\rd}{{\mathrm d}}
\newcommand{\one}{\mathds{1}}

\newcommand{\bB}{{\bf B}}
\newcommand{\bA}{{\bf A}}

\newcommand{\dom}{\operatorname{dom}}
\newcommand{\supp}{\operatorname{supp}}
\newcommand{\curl}{\operatorname{curl}}

\newcommand{\bel}{\begin{equation} \label}
\newcommand{\ee}{\end{equation}}

\definecolor{webred}{rgb}{0.75,0,0}
\definecolor{webgreen}{rgb}{0,0.75,0}
\usepackage[citecolor=webgreen,colorlinks=true,linkcolor=webred]{hyperref}

\title[Ground state of a 3d magnetic hamiltonian]{\large On the ground state of the Laplacian with a magnetic field created by a rectilinear current}
\author{Vincent Bruneau}
\address{Universit\'e Bordeaux 1, IMB, UMR CNRS 5251, 351 cours de la lib\'eration, 33405 Talence Cedex\\
{\it E-mail address:} Vincent.Bruneau@u-bordeaux1.fr}
\author{Nicolas Popoff}
\address{
Laboratoire CPT, UMR 7332 du CNRS, Campus de Luminy, 13288 Marseille cedex 9, France
\\
 {\it E-mail address:} nicolas.popoff@cpt.univ-mrs.fr}

\date{\today}

\begin{document}
\begin{abstract}
We consider the three-dimensional Laplacian with a magnetic field created by an infinite rectilinear current bearing a constant current. The spectrum of the associated hamiltonian is the positive half-axis as the range of an infinity of band functions all decreasing toward 0. We make a precise asymptotics of the band function near the ground energy and we exhibit a semi-classical behavior. We perturb the hamiltonian by an electric potential. \Bk Helped by the analysis of the band functions, we show that for slow decaying potential, an infinite number of negative eigenvalues are created whereas only finite number of eigenvalues appears for fast decaying potential. The power-like decaying potential determining the finiteness of the negative spectrum is different than for the free Laplacian.
\end{abstract} 
\maketitle

\section{Introduction}
%\subsection{Context}
\subsection{Motivation and problematic}
\paragraph{Physical context}
We consider in $\R^3$ the magnetic field created by an infinite rectilinear wire bearing a constant current. Let $(x,y,z)$ be the cartesian coordinates of $\R^3$ and assume that the wire coincides with the $z$ axis. Due to the Biot \& Savard law, the generated magnetic field writes 
$$\bB(x,y,z)=\frac{1}{r^2}(-y,x,0)$$ where $r:=\sqrt{x^2+y^2}$ is the radial distance corresponding to the distance to the wire. Let $\bA(x,y,z):=(0,0,\log r)$ be a magnetic potential satisfying $\curl \bA=\bB$. We define the unperturbed magnetic hamiltonian 
$$H_{\bA}:=(-i\nabla-\bA)^2=D_{x}^2+D_{y}^2+(D_{z}-\log r)^2; \qquad D_{j}:=-i\partial_{j}$$
initially defined on $C_0^\infty(\R^3)$ and then self-adjoint in $L^2(\R^3)$. It is known (see \cite{Yaf03}, and \cite{Yaf08} for a more general setting) that the spectrum of $H_{\bA}$ has a band structure with band functions defined on $\R$ and decreasing from $+ \infty$ toward $0$. Then the spectrum of $H_{\bA}$ is absolutely continuous and coincide with $[0,+ \infty)$. In that case the presence of the magnetic field does not change the spectrum (i.e. $\spec(H_{\bA})= \spec(-\Delta)$), that may be expected since the magnetic field tends to 0 far from the wire. In this article we study the ground state of $H_{\bA}$ and its stability under electric perturbation. These questions are related to the dynamic of spinless quantum particles submitted to the magnetic field $\bB$ and perturbed by an electric potential.

\paragraph{Comparison with the free hamiltonian}
In general the spectrum of a Laplacian may be higher in the presence of a magnetic field (see \cite{AvrHerSi78}). As already said, in our model we still have $\spec(H_{\bA})=\R_{+}$. However the dynamics are very different from the free motion, see \cite{Yaf03} for a description of the classical and quantum dynamics of this model. As we will see, the behavior of the negative spectrum under electrical perturbation is also different that what happens without magnetic field.

 If $V$ is a multiplication operator by a real electric potential $V$ such that $V(H_{\bA}+1)^{-1}$ is compact then the operator $H_{A,V}:= H_{\bA}-V$ is self-adjoint, its essential spectrum coincides with the positive half-axis and discrete spectrum may appear under $0$.  

 Let us recall that, due to the diamagnetic Inequality (see \cite[Section 2]{AvrHerSi78}), the operator $V(H_{\bA}+1)^{-1}$ is compact as soon as $V(-\Delta +1)^{-1}$ is compact. 
 %(in particular for $V\in L^\infty_0:= \{ u \in L^\infty(\R^3); \, u(x,y,z) \longrightarrow 0, \mbox{ as }  \vert(x,y,z)\vert \longrightarrow \infty\} $).
  Moreover, if $\mathcal{N}_{A,V}(\lambda)$ denotes the number of eigenvalues of $H_{\bA}-V$ below $-\lambda<0$, we have (\cite[Theorem 2.15]{AvrHerSi78}):
\bel{Lieb}
\mathcal{N}_{A,V}(0^+)\leq C \int_{\R^3} V_+(x,y,z)^{\frac32} \rd x\rd y\rd z, \qquad V_+:=\max (0,V). 
\ee 
In particular, $H_{\bA}-V$ has a finite number of negative eigenvalues provided that $V_+ \in L^{\frac32}(\R^3)$. But this condition, also valid for $-\Delta -V$, is not optimal in presence of magnetic fields as the results of this article will show. 

We will prove that the discrete spectrum of our operator $H_{\bA}-V$ below 0 is less dense than for $-\Delta -V$ (see Theorem \ref{thm3} and Corollary \ref{cor1}), more precisely for some $V$ the operator $-\Delta -V$ has infinitely many negative eigenvalues whereas $\mathcal{N}_{A,V}(0^+)<+\infty$. In some sense, that means that the absolutely continuous spectrum of $H_{\bA}$ near 0 is less dense that the one of the free Laplacian $-\Delta$. \Bk

\paragraph{Magnetic hamiltonian and band functions}
Several models with constant magnetic field have been studied in the past years. We recall some of them below. In most cases the system has a translation-invariance direction and the magnetic Laplacian is fibered through partial Fourier transform, therefore its study reduces to the study of the band functions that are the spectrum of the fiber operators. The spectrum of the hamiltonian is the range of the band functions (see \cite{GeNi98} for a general setting) and the ground state is given by the infimum of the first band function. The number of eigenvalues created under the essential spectrum by a suitable electric perturbation depends strongly on the shape of the band functions near the ground state as shown on the examples below:

 For the case of a constant magnetic field in $\R^n$, the perturbation by electric potential is described for example in \cite{Sob84} or \cite{Rai90}. When $n=2$, the band functions are constant and equal to the Landau levels.  In \cite{RaiWar02} the authors deal with very fast decaying potential. In that case they prove that the perturbation by an electric potential even compactly supported generates sequences of eigenvalues which converge toward the Landau levels, that is very different from what happens without magnetic field where only a finite number of eigenvalues are created by compactly supported electric perturbation. 
 
In general the band function associated with a Schr\"odinger operator are not constant. The case where the band functions reach their infimum is described in \cite{Rai92} where the author study the  perturbation of a Schr\"odinger operator with periodic electric potential and no magnetic field, whose band functions have non-degenerated minima, providing localization in the phase space. \Bk Let us come back to the case with constant magnetic field. When adding a boundary, the band functions may not be constant anymore. For example when the domain is a two-dimensional infinite strip of finite width with constant magnetic field, it is proved that all the band functions are even with a non-degenerate minimum, see \cite{GeilSen97}.  In \cite{BriRaiSoc08}, the authors investigate the behavior of the spectral shift function near the minima of the band functions, providing the number of eigenvalue created under the ground state when perturbing by an electric potential. Other examples of such a situation is the case of a half-plane with constant magnetic field and Neumann boundary condition, see \cite[Section 4]{BruMirRaik13}, the case of an Iwatsuka model with an odd discontinuous magnetic field, \cite[Section 5]{HisSoc13} and also the case of the Dirichlet Laplacian on a twisted wave guide, \cite{BriKovRaiSoc09}.\Bk

The case of a half-plane with a constant magnetic field and Dirichlet boundary condition is more intriguing and somehow closer to our model: in that case the bottom of the spectrum of the magnetic Laplacian is the first Landau level, but the associated band function does not reach its infimum. In \cite{BruMirRaik13}, the authors gives the precise behavior of the counting function when perturbing by a suitable electric potential. Analog situations based on Iwatsuka models are described in \cite{BruMirRai11} or \cite{HisSoc08I}. 

All the above described situations deal with constant magnetic field.  In this article we deal with a three dimensional variable magnetic field going to 0 far from the $z$-axis and invariant along this axis, therefore the situation is quite different-one may think roughly that the variations of the magnetic field will create non-constant band function as the addition of a boundary does in the case of a constant magnetic field.  Moreover in the above described models the band functions are well separated near the ground state in the sense that the infimum of the second band function is larger than the ground state. In our case there are infinitely many band functions that accumulate toward $\inf\spec(H_{\bA})$, see Figure \ref{F1}, adding a technical challenge when studying the ground state.

In this paper, we give more precise description of the spectrum of $H_{\bA}$ near $0$ with asymptotic expansion of the band functions. Then, we study the finiteness of the number of the negative eigenvalues of $H_{\bA}-V$ for relatively compact perturbations $V$. On one hand, we display  classes of 
%axisymmetric 
potentials giving rise to an accumulation at $0$, of an infinite number of negative eigenvalues, on the other hand, under a decreasing property of $V_+$, we prove the finiteness of the discrete spectrum of $H_{\bA}-V$ below $0$. We obtain a class of polynomially decreasing potentials for which $H_A-V$ has a finite number of negative eigenvalues while the negative spectrum of $-\Delta -V$ is infinite.

\subsection{Main results}

Using the cylindrical coordinates of $\R^3$, we identify $L^2(\R^3)$ with the weighted space $L^2(\R_{+}\times (0,2\pi)\times \R,r \rd r \rd \phi \rd z)$ and the operator $H_{\bA}$ writes: 

%\subsection{Reduction}
%\paragraph{Cylindrical coordinates}
%Let $(r,\phi,z)$ be the cylindrical coordinates of $\R^3$. The operator $H_{\bA}$ writes 
$$H_{\bA}=-\frac{1}{r}\partial_{r}r\partial_{r}-\frac{\partial_{\phi}^2}{r^2}+(\log r-D_{z})^2$$
acting on functions of $L^2(\R_{+}\times (0,2\pi)\times \R,r \rd r \rd \phi \rd z)$.

Let us recall the fibers decomposition of $H_{\bA}$ that can be found with more details in \cite{Yaf03}. We denote by $\cF_{3}$ the Fourier transform with respect to $z$ and $\Phi$ the angular Fourier transform.   We have the direct integral decomposition (see \cite[Section XIII.16]{ReSi78} for the notations about direct decomposition):\Bk
$$ \Phi\cF_{3} H_{\bA}\cF_{3}^{*}\Phi^{*}:=\sum^{\bigoplus}_{m\in \Z}\int_{k\in \R}^{\bigoplus} g_{m}(k) \rd k$$
where the operator 
\bel{D:gmk}
g_{m}(k):=-\frac{1}{r}\partial_{r}r\partial_{r}+\frac{m^2}{r^2}+(\log r-k)^2
\ee
is defined as the extension of the quadratic form
$$q_{m}^{k}(u):=\int_{\R_{+}}\left(|u'(r)|^2+\frac{m^2}{r^2}|u(r)|^2+(\log r-k)^2|u(r)|^2 \right)r \rd r$$
initially defined on $\cC^{\infty}_{0}(\R_{+})$ and closed in $L^2_r(\R_+):=L^2(\R_+,r \rd r)$.

For all $(m,k)\in \Z\times \R$ the operator $g_{m}(k)$ has compact resolvent. 
We denote by $\lambda_{m,n}(k) $ the so-called band functions, i.e. the $n$-th eigenvalue of $g_m(k)$ associated with a normalized eigenvector $u_{m,n}(k)$. 

It is known (\cite{Yaf03}, see also Section \ref{SS:basefiber}) that $k \mapsto \lambda_{m,n}(k)$ is decreasing with  \Bk
$$\lim_{k\rightarrow - \infty} \lambda_{m,n}(k)= + \infty; \qquad \lim_{k\rightarrow + \infty} \lambda_{m,n}(k)= 0.$$
Exploiting semi-classical tools (with semi-classical parameter $h=e^{-k}$, $k>>1$, see Proposition \ref{P:asymptoticvp}), we obtain asymptotic behaviors of the eigenpairs of $g_{m}(k)$ as $k$ tends to infinity. The main result of Section \ref{S:unperturbed} \Bk is the following

\begin{theo}\label{thm1}
For all $(m,n) \in \Z \times \N^*$, there exist constants $C_{m,n}>0$ and $k_{0} \in \R$ such that for all $k \in (k_{0}, + \infty)$,
\bel{band}
|\lambda_{m,n}(k)-(2n-1)e^{-k}+(m^2-\tfrac{1}{4}-\tfrac{n(n-1)}{2})e^{-2k}| \leq C_{m,n}e^{-5k/2}
\ee
\end{theo}
This asymptotics shows that all the band functions tend exponentially to the ground state and cluster according to their energy level, see Figures \ref{F1} and \ref{F2}. \Bk

Let us consider $V$, a multiplication operator such that $V(H_{\bA}+1)^{-1}$ is compact. 
%(for example $V\in L^\infty_0$). 
Considered in $L^2(\R_{+}\times (0,2\pi)\times \R,r \rd r \rd \phi \rd z)$, $V$ is a function of $(r,\varphi, z)$ and it is said {\it axisymmetric} when it does not depend of $\varphi$. 

We want to know how reacts the ground state of $H_{\bA}$ under electrical perturbation.  For potentials slowly decreasing with respect to $r$, we have an infinite number of negative eigenvalues of $H_{\bA}-V$:

\begin{theo}\label{thm2}
Suppose $V$ is a potential such that $V(H_{\bA}+1)^{-1}$ is compact and
\bel{hypinfinite}
V (x,y,z) \geq \langle (x,y) \rangle^{-\alpha} \, v_\perp(z), \qquad \alpha>0.
\ee
\Bk
If $\alpha$ and $ v_\perp$ satisfy one of the assumptions {\it (i)}, {\it (ii)} below, then, $H_\bA-V$ have a infinite number of negative eigenvalues which accumulate to $0$.

{\it (i)} $ \alpha<\frac12$ and $v_\perp \in L^1(\R)$ such that
$$\int_\R v_\perp(z)\rd z >0.$$

{\it (ii)} $v_\perp \geq C \langle z \rangle^{-\gamma}$ with $\gamma>0$ and $\alpha+\frac{\gamma}{2} <1$. 
\end{theo}
The proof uses a construction of quasi-modes based on the eigenfunctions associated with $\lambda_{m,n}(k)$ that leads to a one-dimensional operator in the $z$ variable. The key point is a projection (in the $r$ variable) of the potential $V$ against the eigenfunctions of $g_{m}(k)$ that are localized near the wells of the potential $(\log r-k)^2$ for large $k$.

We also have conditions giving finiteness of the negative spectrum. \Bk

\begin{theo}\label{thm3}
Assume $V$ is a relatively compact perturbation of $H_{\bA}$  such that 
\bel{hypV}
V (x,y,z) \leq \langle (x,y) \rangle^{-\alpha} \, v_\perp(z),
\ee
with $ \alpha>1$ and $v_\perp\in L^p(\R)$ a non negative function with $p\in [1,2]$. 

Then, $H _A-V$ have, at most, a finite number of negative eigenvalues.
\end{theo}

 Let us give some comments concerning the above results in comparison with known borderline behavior of perturbations of the Laplacian. It is not true in general that the number of negative eigenvalues of $-\Delta-V$ is larger than when adding a magnetic field, see Exemple 2 after Theorem 2.15 of \cite{AvrHerSi78}. Theorem \ref{thm2} is a case where the number of negative eigenvalues in presence of magnetic field is infinite as without magnetic field. 

However due to the diamagnetic inequality, one might expect for most cases that the density of negative eigenvalues is more important for $-\Delta -V$ than for $H_A-V$. The above results illustrate this phenomenon, indeed we prove that the borderline behavior of the perturbation determining the finiteness of the negative spectrum of $H_A-V$ is different than for $-\Delta -V$:
% Since  $\langle (x,y) \rangle^{-\alpha}\langle z \rangle^{-\gamma} \geq  \langle (x,y,z) \rangle^{-(\alpha+\gamma)}$,
 
%Thus a natural question concern the optimality  of the condition $\alpha + \frac{\gamma}2 <1$ in Theorem \ref{thm2}. (ET rŽsultats sur $-\Delta -V$ ss les HYP i) du thm2???)

%At the moment, we don't know if this condition is optimal, but in view of Theorem \ref{thm3}, we see that the borderline behavior of the perturbation determining the finiteness of the negative spectrum of $H_A-V$ is different than for $(-\Delta -V)$. In particular, we obtain:

\begin{cor}\label{cor1}
Let $V$ be a measurable function on $\R^3$ that obeys
%$$c \langle (x,y,z) \rangle^{-(\alpha+\gamma)} \leq V(x,y,z) \leq C \langle (x,y) \rangle^{-\alpha}\langle z \rangle^{-\gamma},$$
$$c \langle (x,y) \rangle^{-\alpha}\langle z \rangle^{-\gamma}\leq V(x,y,z) \leq C \langle (x,y) \rangle^{-\alpha}\langle z \rangle^{-\gamma},$$
with $\alpha+\gamma <2$, $\alpha>1$ and $\gamma >\frac12$.

Then the operator $-\Delta -V$ have infinitely many negative eigenvalues while the negative spectrum of $H_A-V$ is finite.
\end{cor}
\begin{proof}
 Since  $\langle (x,y) \rangle^{-\alpha}\langle z \rangle^{-\gamma} \geq  \langle (x,y,z) \rangle^{-(\alpha+\gamma)}$, according to \cite[Theorem XIII.6]{ReSi78} 
%(see also   \cite[Theorem XIII.82]{ReSi78}) 
we know that for $V(x,y,z) \geq  \langle (x,y) \rangle^{-\alpha}\langle z \rangle^{-\gamma}$ with $\alpha+\gamma<2$, the operator $-\Delta - V$ has infinitely many negative eigenvalues. The corollary is then deduced from Theorem \ref{thm3}.

\end{proof}
A natural open question concern the existence of a borderline behavior of $V$ which determine the finiteness of the negative spectrum of $H_A-V$. At the moment we can only say that, if it exists, such borderline potential $V_b$ satisfies:
$$C_- \langle (x,y) \rangle^{-\alpha_-}\langle z \rangle^{-\gamma_-} \leq V_b \leq C_+ \langle (x,y) \rangle^{-\alpha_+}\langle z \rangle^{-\gamma_+},$$
with $0< \alpha_- \leq \max(1-\frac{\gamma_-}2; \frac12)$, $\gamma_->0$ and $\alpha_+>1$, $\gamma_+> \frac12$. \Bk

\subsection{ Organisation of the article}
 In Section \ref{S:unperturbed} we recall basis on the fibers of the operator $H_{\bA}$ and their associated band functions $\lambda_{m,n}(k)$. We give the localization of the associated  eigenfunctions for large $k$ and we prove Theorem \ref{thm1}. We also provide numerical computations of the band functions. In Section \ref{S:3}, we construct quasi-modes for the perturbed operator $H_{\bA}-V$ that leads to study a one-dimensional problem and allows to prove Theorem \ref{thm2}. 
 Based on an uniform lower bound of the band functions, Section \ref{S:Finitenumber} combines the Birman-Schwinger principle with results of Section  \ref{S:unperturbed} to prove Theorem \ref{thm3}. The key point is an estimation of the Hilbert-Schmidt norm of  Birman-Schwinger type operator associated with the perturbed hamiltonian.
\Bk

\section{Description of the 1d problem associated with the unperturbed hamiltonian}
\label{S:unperturbed}
In this section we first recall results from \cite{Yaf03} on the behavior of the band functions $k\mapsto\lambda_{m,n}(k)$. Then we give Agmon estimates on the associated eigenfunctions and we perform an asymptotic expansion of $\lambda_{m,n}(k)$ when $k$ goes to $+\infty$.  In Section \ref{S:3} and \ref{S:Finitenumber} we will use these expansions to analyse the operator $H_{\bA}-V$.  \Bk

Depending on the context we shall work with different operators all unitarily equivalent to the operator $g_{m}(k)$ written in \eqref{D:gmk}. \Bk Table \ref{T:operators} gives a description of these operators and the notations we use.

\subsection{ Semi-classical point of view 
%Basis on the fiber operator
}
\label{SS:basefiber}
\Bk
\paragraph{Global behavior of the band functions}
As in \cite{Yaf03}, we introduce the parameter $$h:=e^{-k}$$ such that $\log r-k=\log (hr)$. The scaling  $\rho=h r$ shows that $g_{m}(k)$ is unitarily equivalent to 
\begin{equation}
\label{D:gh}
\gg_{m}(h):=
-h^2\frac{1}{\rho}\partial_{\rho}\rho\partial_{\rho}+h^2\frac{m^2}{\rho^2}+(\log \rho)^2 \ 
\end{equation}
acting on $L_\rho^2(\R_{+}):=L^2(\R_{+},\rho \rd \rho)$. We denote by $(\mu_{m,n}(h),\gu_{m,n}(\cdot,h))_{n\geq1}$ the normalized eigenpairs of this operator and by $\gq^{m}_{h}$ the associated quadratic form. \Bk We have $\mu_{m,n}(h)=\lambda_{m,n}(k)$ and 
$$\gu_{m,n}(\rho,h)= h \Bk u_{m,n}\left(\frac{\rho}{h},-\log h\right)$$ 
where $u_{m,n}(\cdot,k)$ is a normalized eigenfunction associated with $\lambda_{m,n}(k)$ for $g_{m}(k)$. \Bk Using the min-max principle and the expression \eqref{D:gh}, it is clear that
%Yafaev deduces (see \cite{Yaf03}) that 
$h\mapsto \mu_{m,n}(h)$ is non decreasing on $(0,+\infty)$ and therefore $k\mapsto \lambda_{m,n}(k)$ is non increasing on $\R$. It was already used by Yafaev (see \cite{Yaf03}) who, moreover, \Bk shows (see \cite[Lemma 2.2 \& 2.3]{Yaf03}) that 
$$\lim_{h\to 0}\mu_{m,n}(h)=0 \quad \mbox{and}\quad \lim_{h\to +\infty}\mu_{m,n}(h)=+\infty \, .$$
Note that these results are extended to more general magnetic fields in \cite[Section 3]{Yaf08}.

\paragraph{The fiber operator in an unweighted space}
Sometimes it will be convenient to work in an unweighted Hilbert space on the half-line, therefore we introduce the isometric transformation 
\begin{align*}
\cM\  :\   L^2(\R_{+}, r \rd r) & \longmapsto L^2(\R_{+},\rd r)
\\
 u(r)& \longmapsto \sqrt{r} \, u(r)
\end{align*}
and we define $ \tg_{m}(k):=\cM g_{m}(k) \cM^{*}$. This operator expressed as
\begin{equation}
\tg_{m}(k):=-\partial_{r}^2+\frac{m^2-\frac{1}{4}}{r^2}+(\log r-k)^2 \, ,
\end{equation}
acting on $L^2(\R_{+})$ and its precise definition can be derived from the natural associated quadratic form initially defined on $\cC^{\infty}_{0}(\R_{+})$ and then closed to $L^2(\R_{+})$.

\subsection{Agmon estimates about the eigenpairs of the fiber operator}
We write 
$$\gg_{m}(h)=-h^2\frac{1}{\rho}\partial_{\rho}\rho\partial_{\rho}+V_{h}^{m}$$
with 
$$V_{h}^{m}(\rho):=\log(\rho)^2+h^2\frac{m^2}{\rho^2}\, .$$
Let  $\gq_{m}^h$ denote the natural associated quadratic form.
%Assume that $\Phi$ is such that $e^{\Phi}u_{\lambda}\in \dom(Q^{m}_{h}$) and apply \eqref{E:IMSgh} with $\chi=e^{\phi}$:
%$$\lambda \| e^{\Phi}u_{\lambda}\|_{L^2_{\rho}(\R_{+})}^2=Q_{h}^{m}(e^{\Phi}u_{\lambda})-h^2\||\Phi'|e^{\Phi}u_{\lambda}\|^2_{L^2_{\rho}(\R_{+})}\, ,$$
%we deduce
%\begin{equation}
%\label{E:IMSweight}
%\int_{\R_{+}}\left( h^2|\partial_{\rho}(e^{\Phi}u_{\lambda})|^2+ \left( V_{h}^{m}-h^2|\Phi'|^2-\lambda\right)|e^{\Phi}u_{\lambda}|^2 \right) \rho \rd \rho=0 \, .
%\end{equation}
Assume that $\mu$ is an eigenvalue satisfying $\mu\leq E+O(h)$ with $E\geq0$, the eikonale equation on the Agmon weight $\gphi$ writes 
$$ h^2 |\gphi'|^2=V_{h}^{m}-E$$
that is 
$$|\gphi'(\rho)|^2=\frac{(\log \rho)^2-E}{h^2}+\frac{m^2}{\rho^2} \, .$$
A solution is given by $\gphi_{h}(\rho)/h$ with 
\begin{equation}
\label{D:agmongphi}
\gphi_{h}(\rho):=\left|\int_{1}^{\rho}\sqrt{\left((\log \rho')^2-E+h^2\frac{m^2}{{\rho'}^{2}}\right)_{+}} \rd \rho'\right|
\end{equation}
\Bk

This function provides the general Agmon estimates: 
\begin{prop}
\label{P:Agmonestimates}
Let $E \geq0$ and $C_{0}>0$. For all $\beta\in (0,1)$ there exist $C(E,\beta)>0$ and $h_{0}>0$ such that for all eigenpairs $(\mu,\gu_{\mu})$ of $\gg_{m}(h)$ with $\mu\leq E+C_{0}h$ and $\gu_{\mu}$ that is $L^2_{\rho}$-normalized there holds:
\begin{equation}
\label{E:Agmestimates1}
\forall h\in (0,h_{0}), \quad \|e^{\beta\frac{\gphi_{h}}{h}}\gu_{\mu}\|_{L^2_{\rho}(\R_{+})} \leq C(E,\beta) \ \ \mbox{and} \ \  \gq_{m}^{h}\left(e^{\beta\frac{\gphi_{h}}{h}}\gu_{\mu}\right) \leq C(E,\beta) \, .
\end{equation}
\end{prop}
\begin{proof}
This proposition is an application of the well-known Agmon estimates for 1D Schr\"odinger operators with confining potential. First we have the following identity for any Lipschitz bounded function $\gphi$, see for example \cite{Si82}, \cite{Ag85} or \cite{HeSj85}:
\begin{equation}
\label{E:IMSgh}
\langle \gg_{m}(h)u,e^{2\gphi} u \rangle_{L^2_{\rho}(\R_{+})}=\gq^{m}_{h}(e^{\gphi} u)-h^2\|\gphi' e^{\gphi}u \|_{L^2_{\rho}(\R_{+})}^2 \, .
\end{equation}
In particular when $u=\gu_{\mu}$ is an eigenfunction associated with the eigenvalue $\mu$ we get
\begin{equation}
\label{E:IMSweight}
\int_{\R_{+}}\left( h^2|\partial_{\rho}(e^{\gphi}\gu_{h})|^2+ \left( V_{h}^{m}-h^2|\gphi'|^2-\mu\right)|e^{\gphi}\gu_{h}|^2 \right) \rho \rd \rho=0 \, .
\end{equation}
We now use this identity with $\gphi=\gphi_{h}/h$ where $\gphi_{h}$ is 
%the solution of the eikonale equation 
defined in \eqref{D:agmongphi}. The remain of the proof is classical and can be found with details in \cite[Proposition 3.3.1]{He88} for example. 
%Note that $\gphi_{h}$ is unbounded and a cut-off procedure is required in order to get the final estimates \eqref{E:Agmestimates1}, see for example \cite[Section 2]{BoDauPopRay12}.
\end{proof}
%Let $\cM:u\mapsto \sqrt{\rho}\,u$ from $L^2_{\rho}(\R_{+})$ into $L^2(\R_{+})$ and remind that $\cM H_{m}(h) \cM^{*}=\gh_{m}(h)$ where $$\gh_{m}(h):=-\partial_{\rho}^2+h^2\frac{m^2-\frac{1}{4}}{\rho^2}+\log(\rho)^2, \quad \rho>0 \, .$$
%Remind that $\gv_{h}(\rho):=\sqrt{\rho}\, \gu_{h}(\rho)$ is an $L^2(\R_{+})$ normalized eigenfunction of $\gg_{m}(h)$ associated with the same eigenvalue $\mu$. Proposition \ref{} implies the following:  
%$$\|e^{\beta\frac{\Phi_{h}}{h}}\gv_{h}\|_{L^2(\R_{+})}=\|e^{\beta\frac{\Phi_{h}}{h}}\gu_{h}\|_{L^2_{\rho}} \leq C(E,\beta) \ \ \mbox{and} \ \ \gQ_{m}^{h}(e^{\beta\frac{\Phi_{h}}{h}}\gv_{h}) \leq C(E,\beta) $$ where $\gQ_{h}$ is the quadratic form associated to $\gg_{m}(h)$.
Note that 
$$\gphi_{h}(\rho) \geq \gphi_{0}(\rho)=\left|\int_{1}^{\rho}\sqrt{\left((\log \rho')^2-E\right)_{+}} \,\rd \rho'\right|$$
 that does not depend neither on $m$ nor on $h$. Therefore \eqref{E:Agmestimates1} remains true replacing $\gphi_{h}$ by $\gphi_{0}$ and we get $L^2$ estimates uniformly in $m$, in particular:
\begin{equation}
\label{E:Agmmeg0}
\forall \beta\in (0,1), \forall h\in (0,h_{0}), \quad \|e^{\beta\frac{\gphi_{0}}{h}}\gu_{m,n}(\cdot,h) \|_{L^2_{\rho}(\R_{+})} \leqÊC(E,\beta) 
\end{equation}
for all normalized eigenfunction $\gu_{m,n}(\cdot,h)$ of $\gg_{m}(h)$ associated with any eigenvalue $\mu_{m,n}(h)$ satisfying $\mu_{m,n}(h) \leq E+C_{0}h$ where $C_{0}>0$ is a set constant.

When $E=0$ (that means that we are looking at the low-lying energies) the Agmon distance $\gphi_{0}$ is explicit: 
\begin{align*}
\gphi_{0}(\rho)&=\left|\int_{1}^{\rho}|\log \rho'| \rd \rho'\right|
\\
&=\left|[\rho' \log\rho'-\rho']_{1}^{\rho}\right|=\left|\rho \log \rho-\rho+1\right| \, .
\end{align*}
Let us express %traduce 
this in the original cylindrical variable $r=\frac{\rho}{h}$ with the Fourier parameter $k=-\log h$. The associated Agmon distance writes 
\bel{D:PhiAg}
\Phi_{0}(r,k):=\frac{\gphi_{0}(\rho)}{h}=e^k \gphi_{0}(re^{-k})= r(\log(r)-k) -r+e^{k} \, .
\ee
Writing the previous estimates in these variables we get that for $k$ large enough:
\begin{equation}
\label{E:agmonnormeL2}
\|e^{\beta \Phi_{0}(\cdot,k)}u_{m,n}(\cdot,k) \|_{L^2_{r}(\R_{+})}  \leq C(0,\beta) \ \ \mbox{and} \ \ \|e^{\beta \Phi_{0}(\cdot,k)}\tu_{m,n}(\cdot,k) \|_{L^2(\R_{+})} \leq C(0,\beta)
\end{equation}
where $\tu_{m,n}(r):=\sqrt{r}\, u_{m,n}(\cdot,k)$ is a normalized eigenvector associated with $\lambda_{m,n}(k)$ for the operator $\tg_{m}(k)$ in the unweighted space $L^2(\R_{+})$.

The function $r\mapsto \Phi_{0}(r,k)$ is positive, decreasing on $(0,e^{k})$ and increasing on $(e^{k},+\infty)$. It vanishes when $r=e^{k}$, so we find that the eigenfunction of the operator $g_{m}(k)$ are localized at the minimum of the wells $r=e^{k}$.

\subsection{Asymptotics for the small energy}
In this section we provide an asymptotic expansion of $\mu_{m,n}(h)$ for fixed $(m,n)$ when $h$ goes to 0, namely:

\begin{prop}
\label{P:asymptoticvp}
For all $(m,n)\in \Z\times \N^*$ there exists $C_{m,n}>0$ and $h_{0}>0$ such that 
$$\forall h\in (0,h_{0}), \quad |\mu_{m,n}(h)-(2n-1)h-(m^2-\tfrac{1}{4}-\tfrac{n(n-1)}{2})h^2| \leqÊC_{m,n}h^{5/2}.$$
\end{prop}
%\begin{rem}
%\label{R:asymptotick}
%In the frequency variable $k$ this estimate becomes
%$$\exists k_{0}\in \R, \forall k\in (k_{0},+\infty), \quad |\lambda_{m,n}(k)-(2n-1)e^{-k}-(m^2-\tfrac{1}{4}-\tfrac{n(n-1)}{2})e^{-2k}| \leqÊC_{m,n}e^{-5k/2}$$
%\end{rem}

The operator $\gg_{m}(h)$ written in \eqref{D:gh} is a semiclassical Schr\"odinger operator with a potential which has a unique minimum at $\rho=1$. We will use the technics of the harmonic approximation as described in \cite{DiSj99}, \cite{Si82} or \cite{He88} to derive the asymptotics of the eigenvalues. The remain of this section is devoted to the proof of Proposition \ref{P:asymptoticvp} which implies Theorem \ref{thm1} because $\lambda_{m,n}(k)=\mu_{m,n}(e^{-k})$.

\paragraph{Canonical transformations}
As above we introduce the operator $\tgg_{m}(h):=\cM \gg_{m}(h) \cM^{*}$ in the unweighted space where $\cM: \gu(\rho) \mapsto \sqrt{\rho}\,\gu(\rho)$. We get
$$\tgg_{m}(h)=-h^2\partial_{\rho}^2+h^2\frac{m^2-\frac{1}{4}}{\rho^2}+\log^2 \rho$$
acting on the unweighted space $L^2(\R_{+})$. \Bk
Apply now the change of variable $t=\frac{\rho-1}{\sqrt{h}}$. We get that $\tgg_{m}(h)$ is unitarily equivalent to $h\hgg_{m}(h)$ where
$$\hgg_{m}(h):=-\partial_{t}^2+\frac{\log^2(1+\sqrt{h}t)}{h}+h\frac{m^2-\frac{1}{4}}{(1+\sqrt{h}t)^2}$$
acting on $L^2(I_{h})$ with $I_{h}=(-h^{-1/2},+\infty)$. As we will see below, this operator has a suitable shape to make an asymptotic expansion of its eigenvalues when $h\to 0$. 
 
 \paragraph{Asymptotic expansion and formal construction of quasi-modes}
We write a Taylor expansion of the potential near $t=0$:
\begin{align}
\label{E:taylorpotentiel}
\frac{\log^2(1+\sqrt{h}t)}{h}+h\frac{m^2-\frac{1}{4}}{1+\sqrt{h}t}=t^2-h^{1/2}t^3+(\tfrac{11}{12}t^4+m^2-\tfrac{1}{4})h+R(t,h) 
%=V_{0}(t)+h^{1/2}V_{1}(t)+hV_{2}(t)+R(t,h) 
\end{align}
where $R(t,h)$ will later be controlled by $(1+|t|)^5h^{3/2}$.

We write 
$$\hgg_{m}(h)=L_{0}+h^{1/2}L_{1}+hL_{2}+R(\cdot,h)$$
where 
\begin{equation*}
\left\{
\begin{aligned}
&L_{0}:=-\partial_{t}^2+t^2 \, , 
\\
&L_{1}:=-t^3 \, ,
\\
&L_{2}:=\left(\tfrac{11}{12}t^4+m^2-\tfrac{1}{4}\right) \, .
\end{aligned}
\right.
\end{equation*}
At first we consider these operator as acting on $L^2(\R)$ and we look at a quasi-mode for $L_{0}+h^{1/2}L_{1}+hL_{2}$ defined on $\R$. Using a suitable cut-off function this procedure will provide a quasi-mode for $\hgg_{m}(h)$. 

We look for
%at 
a quasi-mode of the form 
$$(E(h),f(\cdot,h))=(E_{0}+h^{1/2}E_{1}+hE_{2},f_{0}+h^{1/2} f_{1}+hf_{2})\, .$$
We are led to solve the following system: 
\begin{subnumcases}{}
\label{Eqsyst-1}
L_{0}f_{0}=E_{0}f_{0}\, ,\\
\label{Eqsyst-2}
L_{1}f_{0}+L_{0}f_{1}=E_{0}f_{1}+E_{1}f_{0}\, , \\
\label{Eqsyst-3}
L_{2}f_{0}+L_{1}f_{1}+L_{0}f_{2}=E_{2}f_{0}+E_{1}f_{1}+E_{0}f_{2}\, .
\end{subnumcases}
Since $L_{0}$ is the quantum harmonic oscillator, to solve \eqref{Eqsyst-1} we choose for $E_{0}$ the $n$-th Landau level: 
\begin{equation}
\label{E:defE0}
E_{0}:=2n-1, \ \ n \geq 1
\end{equation}
and 
$$f_{0}=f_{0,n}:=\Psi_{n}, \ \ n \geq 1$$
where $\Psi_{n}$ is the $n$-th normalized Hermite's function with the convention that $\Psi_{1}(t)=(2\pi)^{-1/4}e^{-t^2/2}$.

We take the scalar product of \eqref{Eqsyst-2} against $f_{0,n}$ and we find 
$$E_{1}=\langle (L_{0}-E_{0})f_{1},f_{0,n} \rangle +\langle L_{1}f_{0,n},f_{0,n}\rangle=\langle L_{1}f_{0,n},f_{0,n}\rangle \ . $$
Notice that $f_{0,n}$ is either even or odd and that $L_{1}f_{0,n}$ has the opposite parity. Therefore the function $L_{1}f_{0,n}\cdot f_{0,n}$ is odd for all $n\geq1$ and we get 
\begin{equation}
\label{E:defE1}
E_{1}=0 \, .
\end{equation}
We find $f_{1}$ by solving
\eqref{Eqsyst-2}:
\begin{equation}
\label{Eqsyst-2bis}
(L_{0}-E_{0})f_{1}=-L_{1}f_{0,n}=t^3 \Psi_{n}(t) \ . 
\end{equation}
Using $t\Psi_{n}(t)=\sqrt{\frac{n-1}{2}}\Psi_{n-1}(t)+\sqrt{\frac{n}{2}}\Psi_{n+1}(t)$, we write $t^3\Psi_{n}(t)$ on the basis of the Hermite's functions:
$$t^3\Psi_{n}(t)=a_{n}\Psi_{n-3}(t)+b_{n}\Psi_{n-1}(t)+c_{n}\Psi_{n+1}(t)+d_{n}\Psi_{n+3}(t) \ $$
with
\begin{equation}
\label{E:an--dn}
\forall n\geq1, \ \  \, 
\left\{
\begin{aligned}
&a_{n}=2^{-3/2}\sqrt{(n-1)(n-2)(n-3)}\\
&b_{n}=2^{-3/2}3(n-1)\sqrt{n-1}\\
&c_{n}=2^{-3/2}3n\sqrt{n}\\
&d_{n}= 2^{-3/2}\sqrt{n(n+1)(n+2)} \ .
\end{aligned}
\right.
\end{equation}
Therefore the unique solution to \eqref{Eqsyst-2bis} orthogonal to $f_{0,n}$ is:
$$f_{1}=f_{1,n}:=\left(-\frac{a_{n}}{6}\Psi_{n-3}-\frac{b_{n}}{2}\Psi_{n-1}+\frac{c_{n}}{2}\Psi_{n+1}+\frac{d_{n}}{6}\Psi_{n+3}\right) \  $$
with $a_{n}=0$ when $n\leq3$ and $b_{n}=0$ when $n=1$ (see \eqref{E:an--dn}).

We now take the scalar product of \eqref{Eqsyst-3} against $f_{0,n}$: 
\begin{equation}
\label{Compatibilite:E2}
E_{2}=\langle L_{2}f_{0,n},f_{0,n} \rangle+\langle L_{1}f_{1,n},f_{0,n}\rangle  \ .
\end{equation}
Computations provides
$$\langle L_{2}f_{0,n},f_{0,n} \rangle=\left(\tfrac{11}{12}\|t^2 f_{0,n} \|^2+m^2-\tfrac{1}{4} \right)= \left(\tfrac{11}{16}(2n^2-2n+1)+m^2-\tfrac{1}{4} \right)\ $$
and
$$\langle L_{1}f_{1,n},f_{0,n}\rangle =\left(\frac{a_{n}^2}{6}+\frac{b_{n}^2}{2}-\frac{c_{n}^2}{2}-\frac{d_{n}^2}{6} \right)=\frac{1}{16}\left(-30n^2+30n-11 \right)\, , $$
therefore we get
\begin{equation}
\label{E:defE2}
E_{2}=\left(-\frac{n(n-1)}{2}+m^2-\frac{1}{4} \right) \ . 
\end{equation}
We deduce from \eqref{Eqsyst-3}: 
$$(L_{0}-E_{0})f_{2}=E_{2}f_{0,n}-L_{1}f_{1,n}-L_{2}f_{0,n} \, . $$
Since the compatibility condition is satisfies by the choice of $E_{2}$ (see \eqref{Compatibilite:E2}), the Fredholm alternative provides a unique solution $f_{2}=f_{2,n}$ orthogonal to $f_{0,n}$. As above it may be computed explicitly using the Hermite's functions. Notice that $f_{2,n}$ depends on $m$ as $E_{2}$, see \eqref{E:defE2}.

We finally define
$$f_{m,n}(t,h):=f_{0,n}(t)+h^{1/2}f_{1,n}(t)+hf_{2,n}(t)$$

\paragraph{Evaluation of the quasi-mode and upper bound}
The above formal construction provides functions on $\R$ and we will now use a cut-off function in order to get quasi-modes for $\hgg_{m}(h)$. Let $\chi\in \cC^{\infty}_{0}(\R,[0,1])$ be a  cut-off function increasing such that $\chi(t)=0$ when $t \leq -1/2$ and $\chi(t)=1$ when $t \geq -1/4$. We define $\chi(t,h):=\chi(h^{1/2}t)$ and 
$$\hgv_{m,n}(t,h):=\chi(t,h)f_{m,n}(t,h) \, .$$
Recall that $\hgg_{m,n}(h)$ acts on $L^2(I_{h})$ with $I_{h}=(-h^{-1/2},+\infty)$. Since $\supp\left(\hgv_{m,n}(\cdot,h)\right)\subset (-\frac{1}{2}h^{-1/2},+\infty)$ and $\hgv_{m,n}(\cdot,h)$ has exponential decay at $+\infty$, we have $\hgv_{m,n}\in \dom(\hgg_{m}(h))$. Let $$E_{m,n}(h):=E_{0}+h^{1/2}E_{1}+hE_{2}$$
where $E_{0}$, $E_{1}$ and $E_{2}$ are defined in \eqref{E:defE0}, \eqref{E:defE1} and \eqref{E:defE2}. 

We now evaluate $\| \left(\hgg_{m}(h)-E_{m,n}(h)\right)\hgv_{m,n}(\cdot,h)\|_{L^2(I_{h})}$. The procedure is rather elementary but for the sake of completeness we provide details below.
We have 
\begin{multline}
\label{E:decoupageestimqm}
\| \left(\hgg_{m}(h)-E_{m,n}(h)\right)\hgv_{m,n}(\cdot,h)\|_{L^2(I_{h})} \leqÊ\| [\hgg_{m}(h),\chi(\cdot,h)] f_{m,n}(\cdot,h)\|_{L^2(I_{h})}
\\
+\| \chi(\cdot,h)R(\cdot,h)f_{m,n}(\cdot,h) \|_{L^2(I_{h})}
+\|\chi(\cdot,h)\left(L_{0}+h^{1/2}L_{1}
+hL_{2} -E_{m,n}(h)\right)f_{m,n}(\cdot,h) \|_{L^2(I_{h})}
\end{multline}
We have $ [\hgg_{m}(h),\chi(\cdot,h)]f_{m,n}(t,h)=-2h^{1/2}\chi'(h^{1/2}t)f_{m,n}'(t,h)-h\chi''(h^{1/2}t)f_{m,n}(t,h)$ therefore $t\mapsto [\hgg_{m}(h),\chi(\cdot,h)]f_{m,n}(t,h)$ is supported in $[-\frac{1}{2}h^{-1/2},-\frac{1}{4}h^{-1/2}]$ and since $f_{m,n}(\cdot,h)$ and $f_{m,n}'(\cdot,h)$ have exponential decay we get 
\begin{equation}
\label{E:estimCommutator}
 \| [\hgg_{m}(h),\chi(\cdot,h)] f_{m,n}(\cdot,h)\|_{L^2(I_{h})}= \cO(h^{\infty})\, .
 \end{equation}
Remind that $R(t,h)$ is defined in \eqref{E:taylorpotentiel}, we get 
$$\exists C>0, \forall h>0, \forall t\in \supp(\chi(\cdot,h)), \quad |R(t,h)| \leq C h^{3/2}(1+|t|)^5 .$$
Using the exponential decay of $f_{m,n}$ we get $C_{m,n}>0$ such that 
\begin{equation}
\label{E:estimereste}
\| \chi(\cdot,h)R(\cdot,h)f_{m,n}(\cdot,h) \|_{L^2(I_{h})} \leq C_{m,n}h^{3/2}.
\end{equation}
The last term of \Gr \eqref{E:decoupageestimqm} \Bk %\eqref{E:defE2}
 is easily computed: 
$$
\left(L_{0}+h^{1/2}L_{1}
+hL_{2} -E_{m,n}(h)\right)f_{m,n}(\cdot,h)=h^{3/2}\left((L_{1}-E_1)f_{2,n}+(L_{2}-E_2)f_{1,n} \right)+h^2L_{2}f_{2,n}
$$
and we get $C_{m,n}>0$ such that
$$\|\chi(\cdot,h)\left(L_{0}+h^{1/2}L_{1}
+hL_{2} -E_{m,n}(h)\right)f_{m,n}(\cdot,h) \|_{L^2(I_{h})} \leq C_{m,n} h^{3/2}.$$
Combining this with \eqref{E:estimCommutator} and \eqref{E:estimereste} in \eqref{E:decoupageestimqm} we get 
\begin{equation}
\label{E:estimationqm}
\exists C_{m,n}, \exists h_{0}>0, \forall h\in (0,h_{0}), \quad \| \left(\hgg_{m}(h)-E_{m,n}(h)\right)\hgv_{m,n}(\cdot,h)\|_{L^2(I_{h})} \leq C_{m,n}h^{3/2}. 
\end{equation}
Moreover we have 
\begin{align*}
\| \hgv_{m,n}(\cdot,h)\|_{L^2(I_{h})} &= \| f_{m,n}(\cdot,h)\|_{L^2(\R)}+\cO(h^{\infty})
\\
&=
\| f_{0,n}\|_{L^2(\R)}+\cO(h^{1/2})
\\
&=1+\cO(h^{1/2})
\end{align*}
where the above estimates depends on $(m,n)$. Since $\gg_{m}(h)$ is unitarily equivalent to $h\hgg_{m}(h)$, $\mu_{m,n}(h)/h$ is the $n$-th eigenvalue of $\hgg_{m}(h)$ and the spectral theorem applied to \eqref{E:estimationqm} shows that 
\begin{equation}
\label{E:upperboundcanonique}
\exists C_{m,n}, \exists h_{0}>0, \quad \frac{\mu_{m,n}(h)}{h} \leq E_{m,n}(h)+C_{m,n}h^{3/2}
\end{equation}
and we have proved the upper bound of Proposition \ref{P:asymptoticvp}.

\paragraph{Arguments for the lower bound}
The complete procedure for the proof of the lower bound of the eigenvalues of $\hgg_{m}(h)$ using the harmonic approximation can be found in \cite[Chapter 4]{DiSj99} or \cite[Chapter 3]{He88}. We recall here the main arguments. Let 
$$\widehat{\Phi}_{0}(t,h):= (1+\sqrt{h}t)\log(1+\sqrt{h}t)-\sqrt{h}t $$
be the distance of Agmon in the $t$-variable, the estimates provided in \eqref{E:Agmmeg0} becomes: 
$$\forall \beta\in (0,1),\quad  \|e^{\beta\frac{\widehat{\Phi}_{0}}{h}} \hgu_{m,n}(\cdot,h) \|_{L^2(I_{h})} \leq C(E,\beta) $$
where $\hgu_{m,n}(\cdot,h)$ is the $n$-th eigenvector associated to $\hgg_{m}(h)$. Therefore there holds \it a priori \rm estimates on the eigenfunctions proving that they concentrate near $t=0$ when $h$ tends to 0. Using a Grushin procedure (see \cite{Gru73}), these eigenfunctions are used as quasi-modes for the first order approximation $L_{0}$ and this provides a rough lower bound on the eigenvalues $\frac{\mu_{m,n}(h)}{h}$ of $\hgg_{m}(h)$ by the eigenvalues of $L_{0}$ that are the Landau levels, modulo some remainders. Combining this with \eqref{E:upperboundcanonique}, we get that there are gaps in the spectrum of $\hgg_{m}(h)$ and the spectral theorem applied to \eqref{E:estimationqm} proved the lower bound on $\frac{\mu_{m,n}(h)}{h}$ and therefore the lower bound of Proposition \ref{P:asymptoticvp}.

%\Gr $H(k)$ est-il vraiment utilis\'e???. Son spectre est discret, comme $g_m(k)$. \Bk

%\subsection{Table (ne sera pas une sous-section)}
%\

\begin{table}[ht!]
\begin{center}
{\small
\renewcommand{\arraystretch}{1.5}
\begin{tabular}{| p{1.3cm}|p{5.9cm}|p{1.9cm}|p{0.8cm}|p{4cm}|}
%\begin{tabular}{|p{2pt}|c|c|c|c|}
  \hline
  Notation & Operator & Space & Form  & Eigenpairs  \\
  \hline
  \hline
$H_{\bA}$ & $(-i\nabla-\bA)^2$ & $L^2(\R^3)$ & \hrulefill  & spectrum$=\R_{+}$
%\\
%\hline
%$H(k)$ & $-\Delta_{x,y}+(\log r-k)^2$ & $L^2(\R^2)$ &  \hrulefill &  %spectrum$=[\lambda_{0,0}(k),+\infty)$
\\
\hline
$g_{m}(k)$ & $-\frac{1}{r}\partial_{r}r\partial_{r}+\frac{m^2}{r^2}+(\log r-k)^2$ & $L^2(\R_{+}, r \rd r)$ & $q_{m}^{k}$ &  $(\lambda_{m,n}(k),u_{m,n}(r,k))$
\\
\hline
$\tg_{m}(k)$ & $-\partial_{r}^2+\frac{m^2-\frac{1}{4}}{r^2}+(\log r-k)^2$ & $L^2(\R_{+}, \rd r)$ & $\tq_{m}^{k}$ &  $(\lambda_{m,n}(k),\tu_{m,n}(r,k))$
\\
\hline
$\gg_{m}(h)$ & $-h^2\frac{1}{\rho}\partial_{\rho}\rho\partial_{\rho}+h^2\frac{m^2}{\rho^2}+(\log \rho)^2$ & $L^2(\R_{+}, \rho \rd \rho)$ & $\gq_{m}^{h}$ &  $(\mu_{m,n}(h),\gu_{m,n}(\rho,h))$
\\
\hline
$\tgg_{m}(h)$ & $-h^2\partial_{\rho}^2+h^2\frac{m^2-\frac{1}{4}}{\rho^2}+(\log \rho)^2$ & $L^2(\R_{+}, \rd \rho)$ & $\tgq_{m}^{h}$ &  $(\mu_{m,n}(h),\tgu_{m,n}(\rho,h))$
\\
\hline
$\hgg_{m}(h)$ & $-\partial_{t}^2+h\frac{m^2-\frac{1}{4}}{(1+h^{1/2}t)^2}+(\log (1+h^{1/2}t))^2$ & $L^2(I_{h}, \rd t)$ & $\hgq_{m}^{h}$ &  $(h^{-1}\mu_{m,n}(h),\hgu_{m,n}(t,h))$
\\
\hline
\end{tabular}}
\end{center}
\caption{\label{T:operators} Operators and notations. Remind that  $\rho=hr$ with $r=\sqrt{x^2+y^2}$, $h=e^{-k}$ and $I_{h}=(-h^{-1/2},+\infty)$.}
\end{table}

\subsection{Numerical approximation of the band functions}
We use the finite element library Melina (\cite{Melina++}) to compute numerical approximations of the band functions $\lambda_{m,n}(k)$ with $0\leq m \leq 2$ and $1 \leq n \leq 4$. For $k\in [-2,6]$), the computations are made on the interval $[0,L]$ with $L$ large enough and an articifial Dirichlet boundary condition at $r=L$. According to the decay of the eigenfunctions provided by the Agmon estimates we have chosen $L=2e^{6}$ so that the region $\{r\sim e^{k} \}$ where are localized the associated eigenfunction is included in the computation domain.

On Figure \ref{F1} we have plot the numerical approximation of $\lambda_{m,n}(k)$ for the range of parameters described above. According to the theory, they all decrease from $+\infty$ toward 0. Notice that the band functions may cross for different values of $m$.

 On figure \ref{F2} we have zoomed on the lowest energies $\lambda <<1$ and we have also plotted the first order asymptotics $k\mapsto (2n-1)e^{-k}$. We see that for set $1 \leqÊn \leq 4 $, the band functions $\lambda_{m,n}(k)_{0 \leq m \leq 2}$ cluster around the first order asymptotic $(2n-1)e^{-k}$ according to Theorem \ref{thm1}.

\begin{figure}[ht]
\begin{center}
\includegraphics[keepaspectratio=true,width=15cm]{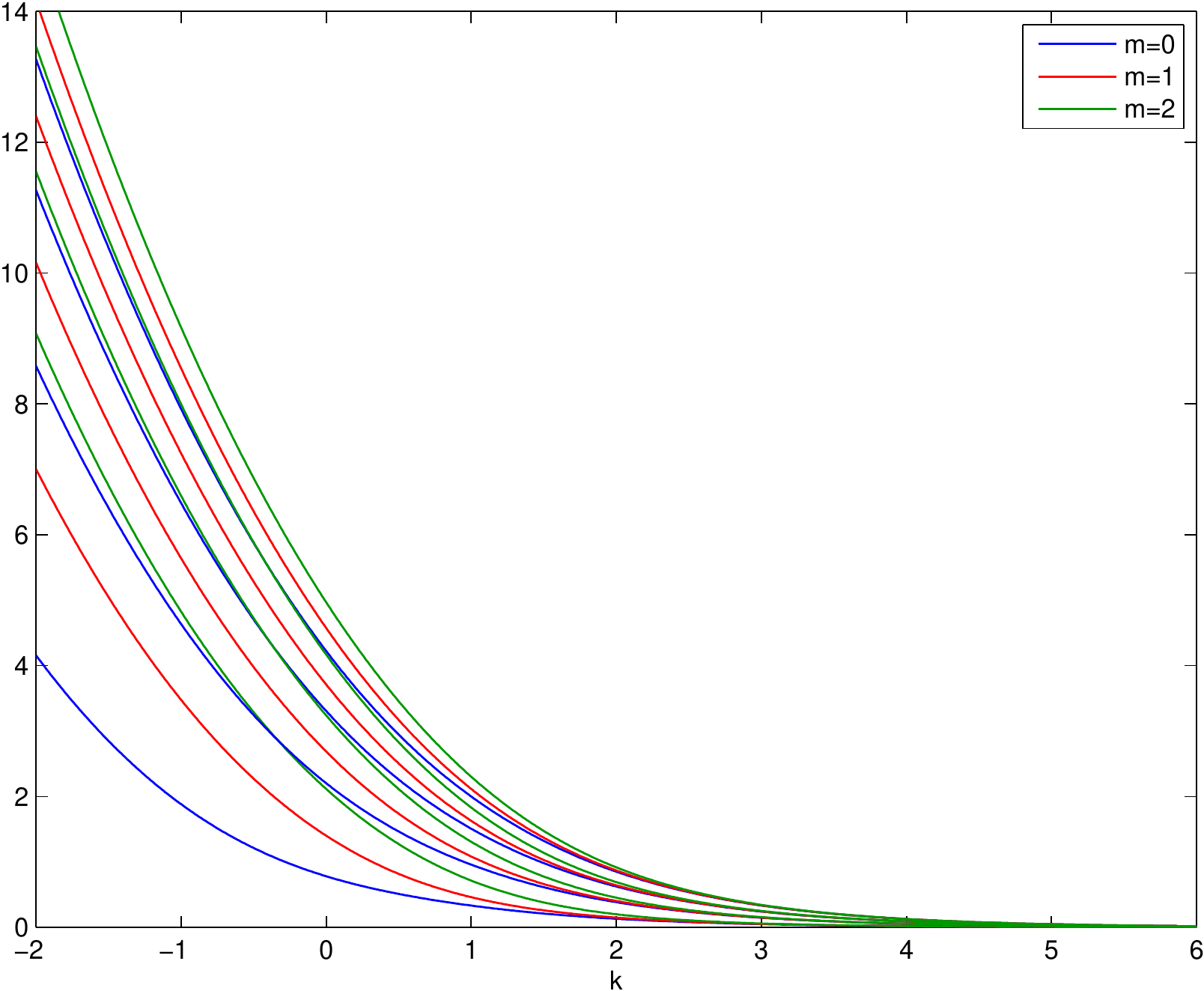}
\caption{The band functions $\lambda_{m,n}(k)$ for $0 \leq m \leq 2$ and $1 \leq n \leq 4$ and $k\in [-2,6]$.}
\label{F1}
\end{center}
\end{figure}

\begin{figure}[ht]
\begin{center}
\includegraphics[keepaspectratio=true,width=15cm]{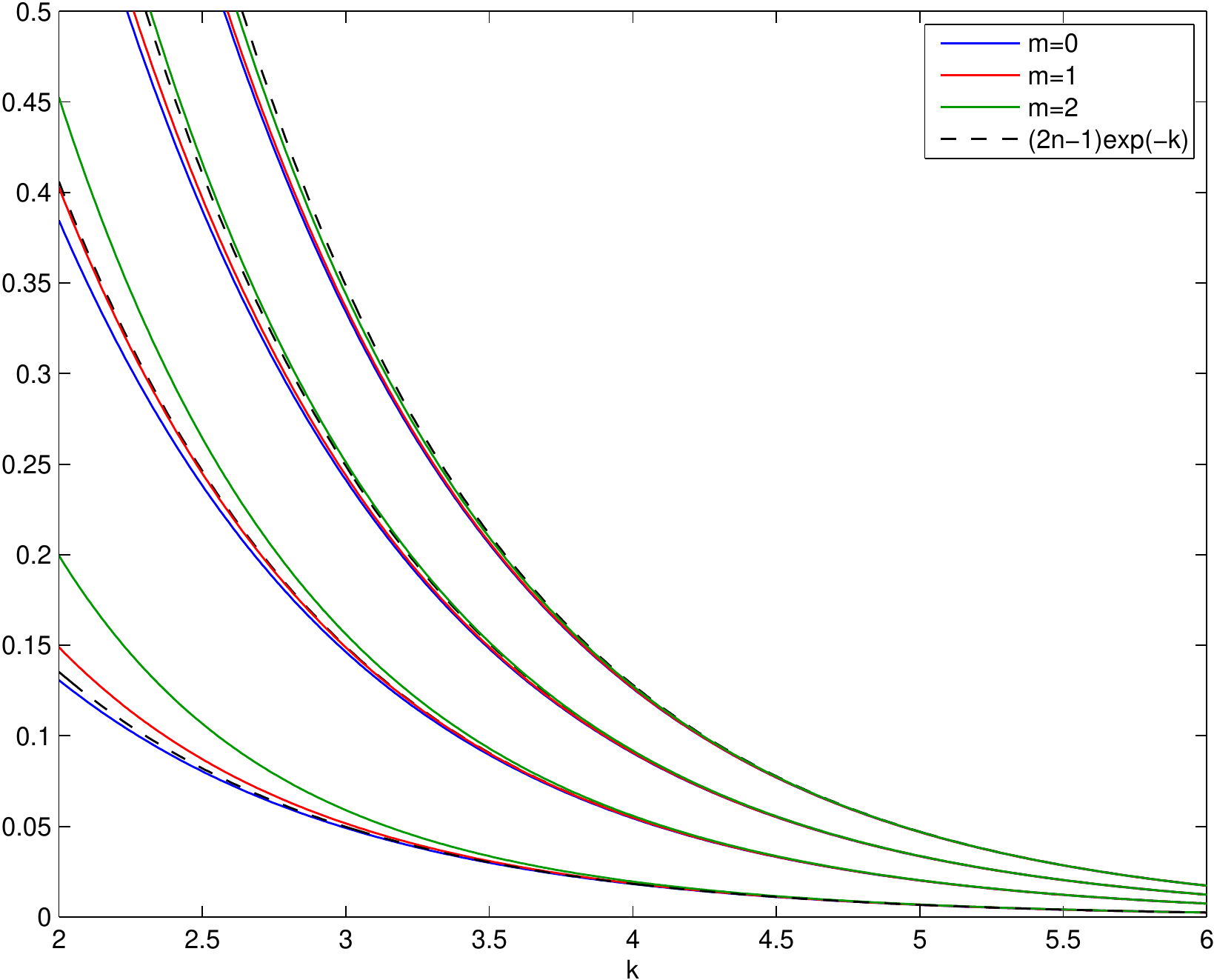}
\caption{Zoom on the lowest energies compared with the first order asymptotics $(2n-1)e^{-k}$. Each cluster corresponds to an energy level $n$.}
\label{F2}
\end{center}
\end{figure}

%\newpage

\section{Construction of quasi-modes and infiniteness of negative eigenvalues}\label{S:3}

In this section we prove Theorem \ref{thm2} giving infinitely many eigenvalues below $0$ for a slowly decreasing perturbation.

First, we consider $V$ depending only on $(r,z)$ and \Bk we construct quasi-modes which allow to reduce the existence of infinitely many negatives eigenvalues  to the existence of sufficiently small eigenvalues of some 1D-effective problems $D^2_z-V_{m,n}$. Then, we study the effective potential $V_{m,n}$ and conclude the proof of Theorem \ref{thm2}.

%\section{Construction of quasi-modes for the perturbed hamiltonian} \label{S:3}
\subsection{Quasi-modes}
We construct quasi-modes for the perturbed operator $H_{\bA}-V$ where $V$ is axisymmetrical. Let 
$$\psi_{m,n}(r,\phi,z,k):=e^{im\phi}e^{ikz}u_{m,n}(r,k)f(z)$$
where $f\in L^2(\R)$, $(m,n,k)$ will be chosen later and $u_{m,n}(r,k)$ is a normalized eigenfunction of $g_{m}(k)$ associated with $\lambda_{m,n}(k)$. We have:
\begin{lem}\label{lemEffective}
For any $\epsilon>0$, 
\bel{InegEffective}
\langle  (H_{\bA}-V)\psi_{m,n},\psi_{m,n} \rangle
%_{L^2(\R^+\times \R,r \rd r \rd z)} 
\leq (1+\epsilon) \lambda_{m,n}(k) \|f\|_{L^2(\R)}^2
+(1+\epsilon^{-1})\parallel D_{z}f\parallel^2_{L^2(\R)}- \langle  V_{m,n}(.,k)f ,f \rangle_{L^2(\R)}
\ee
with 
\bel{DefVmn}
V_{m,n}(z,k):=\int_{r}|\tu_{m,n}(r,k)|^2V(r,z)\rd r; \qquad\tu_{m,n}(r,k):=\sqrt{r}\,u_{m,n}(r,k). 
\ee
\end{lem}
\begin{proof}
We have
\begin{multline}
H_{\bA}\psi_{m,n}(r,\phi,z,k)=e^{im\phi}e^{ikz}f(z)g_{m}(k)u_{m,n}(r,k)\\
+
e^{im\phi}e^{ikz}u_{m,n}(r,k)\left(D_{z}^2f+2(\log r-k)D_{z}f(z)\right) \ ,
\end{multline}
that is 
\begin{multline}
(H_{\bA}-V)\psi_{m,n}(r,\phi,z,k)=\lambda_{m,n}(k)\psi_{m,n}(r,\phi,z,k)\\
+
e^{im\phi}e^{ikz}u_{m,n}(r,k)\left(D_{z}^2f+2(\log r-k)D_{z}f(z)-V(r,z)f(z)\right) \ . 
\end{multline}
\begin{multline}
 (H_{\bA}-V)\psi_{m,n}\cdot\overline{\psi_{m,n}} =\lambda_{m,n}(k)u_{m,n}(r,k)^2f(z)^2+
\\ u_{m,n}(r,k)^2\Big(D_{z}^2f(z)+2(\log r -k)D_{z}f(z)-V(r,z)f(z)\Big)\overline{f(z)}.
\end{multline}
%Assume that $\|f\|_{L^2(\R)}=1$ and 
Integrating over $(r,z)$ in the weighted space $(\R_+\times \R, r \rd r \rd z)$ we get
\begin{multline}\label{eq4.6}
\langle  (H_{\bA}-V)\psi_{m,n},\psi_{m,n} \rangle_{L^2(\R_+\times \R,r \rd r \rd z)}=\lambda_{m,n}(k)\|f\|_{L^2(\R)}^2
\\
+\|D_{z}f\|^2+2\int_{r,z}(\log r-k)| u_{m,n}(r,k)|^2 D_{z}f(z) \overline{f(z)} r \rd r \rd z -\int_{z}V_{m,n}(z,k)|f(z)|^2 \rd z.
\end{multline}
Then, using that for any $\epsilon >0$, 
$$| 2 (\log r-k) D_{z}f(z) \overline{f(z)}| \leq \epsilon (\log r-k)^2  |f(z)|^2 + \epsilon^{-1} |D_z f |^2,$$
we deduce,
$$\langle  (H_{\bA}-V)\psi_{m,n},\psi_{m,n} \rangle
\leq  \lambda_{m,n}(k) \|f\|_{L^2(\R)}^2
+(1+\epsilon^{-1})\parallel D_{z}f\parallel^2_{L^2(\R)} $$
$$+ \epsilon \int_{r,z} (\log r-k)^2 |u_{m,n}(r,k)|^2 |f(z)|^2 r \rd r \rd z 
- \langle  V_{m,n}(.,k)f ,f \rangle_{L^2(\R)} .
$$
Since in the sense of quadratic form in $L^2(\R_+\times \R,r \rd r \rd z)$, we have $(\log r-k)^2 \leq g_{m}(k)$, we obtain \eqref{InegEffective} using again that $g_{m}(k) u_{m,n}(r,k)= \lambda_{m,n}(k) u_{m,n}(r,k) $.
%where we have set
%$$V_{m,n}(z,k):=\int_{r}|\tu_{m,n}(r,k)|^2V(r,z)\rd r$$
%and $\tu_{m,n}(r,k):=\sqrt{r}u_{m,n}(r,k)$. 
\end{proof}
\begin{rem}
According to the Feynman-Hellmann formula, the third term in the right hand side of \eqref{eq4.6} is related to the derivative of $\lambda_{m,n}(k)$:
$$\lambda'_{m,n}(k)= -2\int_{r,z}(\log r-k)| u_{m,n}(r,k)|^2  r \rd r .$$
This quantity could be studied more carefully as in \cite{HisPopSoc13} where it is done for another fibered operator, but here, we need only some rough estimates.
\end{rem}
%We write for all $\epsilon \in (0,1)$:
%\begin{align*}
%\left| 2\int_{r,z}(\log r-k)|\psi_{m,n}(r,k)|^2 D_{z}f(z) \overline{f(z)}\rd z \rd r\right|
%& \leq
% \epsilon^{-1}\int_{r}(\log r-k)^2|\psi_{m,n}(r,k)|^2\rd r+\epsilon\|D_{z}f\|^2
% \\
% & \leq \epsilon^{-1}\lambda_{m,n}(k)+\epsilon\|D_{z}f\|^2
%\end{align*}
%Therefore we get 
%$$
%\langle  (H_{\bA}-V)\psi_{m,n},\psi_{m,n} \rangle_{L^2(\R^+\times \R,r \rd r, \rd z)}\leq(1+\epsilon^{-1})\lambda_{m,n}(k)
%+(1+\epsilon)\|D_{z}f\|^2-\int_{z}V_{m,n}(z,k)|f(z)|^2 \rd z
%$$
%Therefore we have to prove that the operator $(1+\epsilon)D_{z}^2-V_{m,n}(\cdot,k)$ has an eigenvalue below $-(1+\epsilon^{-1})\lambda_{m,n}(k)$. We will choose $k$ large such that $\lambda_{m,n}(k)$ is exponentially small and we will describe the behavior of $z\mapsto V_{m,n}(z,k)$.
\subsection{Estimate on the reduced potential}\label{sSecEffective}
We are looking at the asymptotic behavior of the 1D potential $z\mapsto V_{m,n}(z,k)$ by using the localization properties of the eigenfunctions $\tu_{m,n}(\cdot,k)$ when $k$ goes to $+\infty$. In this section all the Landau's notations refers to an asymptotic behavior when $k$ goes to $+\infty$. Set $(m,n)\in \Z \times \N^{*}$, $C_{m,n}>2n-1$ and choose $k$ large enough such that $\lambda_{m,n}(k) \leq C_{m,n}e^{-k}$ (see Theorem \ref{thm1}). Write that 
$\R=I_{k}\cup \complement I_{k}$ with $I_{k}=[e^{k}-a(k),e^{k}+a(k)]$ and $a(k)=o(e^{k})$ will be chosen later. We use \eqref{E:agmonnormeL2} with $E=0$:
$$\int_{\complement I_{k}}|\tu_{m,n}(r,k)|^2  \rd r\leq C(0,\beta)\sup_{r\in \complement I_{k}}e^{-\beta\Phi_{0}(r,k)} $$
where the Agmon distance $\Phi_{0}$ is defined in \eqref{D:PhiAg}. \Bk Since $\Phi_{0}(\cdot,k)$ is decreasing on $(0,e^{k})$ and increasing on $(e^{k},+\infty)$ we have 
$$\inf_{\complement I_{k}}\Phi_{0}(\cdot,k)=\min(\Phi_{0}(e^{k}\pm a(k)) \, .$$
An asymptotic expansion at these points provides 
$$\Phi_{0}(e^{k}\pm a(k),k) \underset{k\to +\infty}{=}\frac{1}{2}a^2(k)e^{-k}+O(a(k)^3e^{-2k}) \, . $$ 
Assume that 
\begin{equation}
\label{E:choixa(k)}
\lim_{k\to +\infty} a^2(k)e^{-k}= +\infty \ \ \mbox{and} \ \ \lim_{k\to +\infty} a^3(k)e^{-2k}=0
\end{equation} then we have 
$$e^{-\beta\Phi_{0}(e^{k}\pm a(k),k)}\underset{k\to +\infty}{\sim} e^{-\frac{\beta}{2}a(k)^2e^{-k}}$$
The condition \eqref{E:choixa(k)} is valid for any $a(k)$ satisfying 
$$e^{\frac{k}{2}} << a(k) << e^{\frac{2k}{3}} $$
and for such an $a(k)$ we get 
\begin{equation}
\label{E:estimationsup}
\sup_{r\in \complement I_{k}}e^{-\beta\Phi_{0}(r,k)} \underset{k\to +\infty}{\sim} e^{-\frac{\beta}{2} a(k)^2e^{-k}}\, . 
\end{equation}
We have 
\begin{align*}
V_{m,n}(z,k) &\geq \inf_{r \in I_{k}}V(r,z)\int_{I_{k}}|\tu_{m,n}(r,k)|^2 \rd r
\\
&\geq  \inf_{r \in I_{k}}V(r,z)(1-C(0,\beta)\sup_{r\in \complement I_{k}}e^{-{\beta}\Phi_{0}(r,k)})
\end{align*}
where we have used $\|\tu_{m,n}(\cdot,k)\|_{L^2(\R_{+})}=1$.

Set $\beta\in (0,1)$ once for all. Choose $\epsilon>0$. Then we deduce from the choice of $a(k)$ in \eqref{E:choixa(k)} and \eqref{E:estimationsup} that there exists $k_{0}$ that depends \it a priori \rm of $(m,n)$ such that 
\bel{MinVmn}
\forall k\geq k_{0}, \forall z\in \R, \quad V_{m,n}(z,k) \geq (1-\epsilon)\inf_{r \in I_{k}}V(r,z)
\ee

\subsection{Proof of Theorem \ref{thm2}} According to the min-max principle, since $V$ satisfies \eqref{hypinfinite}, it is sufficient to prove the infinity of the negative eigenvalues for the axisymmetric potential $V(r,z)=\langle r \rangle^{-\alpha} v_\perp(z)$. \Bk  Let us denote $H^m_{\bA}$ the restriction of $H_{\bA}$ to $e^{im\phi}L^2(\R_{+}\times \R, r \rd r\rd z) $. 
For $V$ axisymmetric, $H_{\bA}-V$ is unitarily equivalent to $\oplus_{m\in \Z} (H^m_{\bA}-V)$, then $H_{\bA}-V$ has infinitely many negative eigenvalues provided that
\begin{itemize}
\item
Either $H^m_{\bA}-V$ has at least one's for all $m\in \Z$,
\item or there exists $m\in \Z$ such that $H^m_{\bA}-V$ has infinitely many negative eigenvalues.
\end{itemize}
Thanks to the min-max principle, Lemma \ref{lemEffective}, implies that for each $m \in \Z$ the number of negative eigenvalues of  $H^m_{\bA}-V$ is at least the number of eigenvalues of $(1+\epsilon^{-1})D^2_z -V_{m,n}(.,k)$ below $- (1+\epsilon) \lambda_{m,n}(k)$, that is the number of eigenvalues of $D^2_z -\frac{\epsilon}{1+\epsilon}V_{m,n}(.,k)$ below $- \epsilon \lambda_{m,n}(k)$. 

For $V(r,z)=\langle r \rangle^{-\alpha} v_\perp(z)$, \Bk the inequality \eqref{MinVmn} implies:
$$\forall k\geq k_{0}, \forall z\in \R, \quad V_{m,n}(z,k) \geq  C e^{-\alpha k} v_\perp(z),$$
and choosing $k$ large enough such that $\lambda_{m,n}(k) \leq C_{m,n}e^{-k}$, we deduce that the number of negative eigenvalues of  $H^m_{\bA}-V$ is at least the number of eigenvalues of 
$$D^2_z -\tfrac{C\epsilon}{1+\epsilon}e^{-\alpha k}v_\perp$$ below $-\epsilon C_{m,n}e^{-k}$.
Then Theorem \ref{thm2} follows by applying the following lemmas (Lemma \ref{lem1D} and Lemma \ref{lem1Dbis}), for $ k$ sufficiently large with $\Lambda (k)= e^{-\alpha k}$, $v=\frac{C\epsilon}{1+\epsilon}v_\perp$ and $\lambda(k)=\epsilon C_{m,n}e^{-k}$.\Bk

\subsection{Lemmas on negative eigenvalues for a family of some 1D Schr\"odinger operators.}
\begin{lem}\label{lem1D}
Let $h(k)= D_z^2 - \Lambda(k) v(z)$ on $\R$, $k \in \R$ with:
$$ \qquad v \in L^1(\R); \qquad \int_\R v(z)\rd z >0, \qquad \Lambda(k)> 0.$$

Let $\lambda(k)$ be a positive function of $k \in \R$ such that 
\bel{hyplambda}
\lim_{k \rightarrow + \infty}\lambda(k) = 0;
\qquad \lim_{k \rightarrow + \infty}\frac{\lambda(k)}{\Lambda(k)^2}=0.
\ee

Then, for $k$ sufficiently large, $h(k)+\lambda(k)$ has at least one negative eigenvalue.

\end{lem}

\begin{proof}
Let us introduce the $L^2-$normalized function
$$ v_k(z):= a(k)^{\frac12}e^{-a(k)| z |}$$
with $a(k)$ satisfying $\lim_{k \rightarrow + \infty}a(k) = 0$ and to be chosen. 
We use $v_k(z)$ as a quasi-mode:
$$\langle h(k) v_k, v_k \rangle = a(k)^2 - \Lambda(k) a(k) \int_\R v(z) e^{-2a(k)| z |}\rd z.$$
Since 
$$\lim_{k \rightarrow + \infty}  \int_\R v(z) e^{-2a(k)| z |}\rd z =  \int_\R v(z) \rd z>0,$$
for $k$ sufficiently large, there exists $C>0$ such that:
$$\langle h(k) v_k, v_k \rangle \leq  a(k)^2 - C \Lambda(k) a(k).$$

By using the min-max principle, it remains to chose $a(k) $ such that $a(k)^2 - C \Lambda(k) a(k) < - \lambda(k)$.
 Under the assumption \eqref{hyplambda}, the polynomial $X^2 - C \Lambda(k) X + \lambda(k)$ has two real roots
$ a_+(k)>a_-(k)>0$ with $a_{-}(k)\leq \frac{2 \lambda(k)}{C \Lambda(k)}$ tending to $0$ as $k$ tends to infinity, see \eqref{hyplambda}. Then, there exists $a(k) $ such that, for $k$ sufficiently large,
$$\langle h(k) v_k, v_k  \rangle <  - \lambda(k),$$
and Lemma \eqref{lem1D} holds.
\end{proof}

\begin{lem}\label{lem1Dbis}
Let $h(k)= D_z^2 - V_k$ on $\R$, $k \in \R$ with $V_k$ satisfying:
$$V_k(z) \geq \Lambda(k) \langle z \rangle^{-\gamma}; \qquad  \gamma \in (0,2); \qquad \Lambda(k)\in (0,1).$$

Let $\lambda(k)$ be a positive function of $k \in \R$ such that 
\bel{hyplambdabis}
%\lim_{k \rightarrow + \infty}\lambda(k) = 0;
\qquad \lim_{k \rightarrow + \infty}\frac{\lambda(k)}{\Lambda(k)^{\frac{2}{2-\gamma}}}=0.
\ee

Then, for $k$ sufficiently large, $h(k)+\lambda(k)$ has at least one negative eigenvalue and the number of negative eigenvalues tends to infinity, as $k$ tends to infinity.

\end{lem}

\begin{proof}
Using the change of variable $y=\Lambda(k)^{\frac{1}{2-\gamma}} z$, it is clear that $h(k)$ is unitarily equivalent to $\Lambda(k)^{\frac{2}{2-\gamma}} \tilde{h}(k)$ with 
 $$\tilde{h}(k):= D_y^2 - \frac{1}{\Lambda(k)^{\frac{2}{2-\gamma}}}V_k\left(\frac{y}{\Lambda(k)^{\frac{1}{2-\gamma}}}\right) .$$
 By assumption on $V_k$, we have:
 $$\frac{1}{\Lambda(k)^{\frac{2}{2-\gamma}}}V_k\left(\frac{y}{\Lambda(k)^{\frac{1}{2-\gamma}}}\right)\geq (\Lambda^{\frac{2}{2-\gamma}}(k)+ y^2)^{-\frac{\gamma}{2}} \geq (1+ y^2)^{-\frac{\gamma}{2}}$$
 where we have used $\Lambda(k)\in (0,1)$. Then the min-max principle implies that the number of negative eigenvalues of $h(k)+\lambda(k)$ is larger that the number of eigenvalues of $D_y^2 - \langle y \rangle^{-\gamma}$ below $-\frac{\lambda(k)}{\Lambda^{\frac{2}{2-\gamma}}(k)}$. Since $\gamma < 2$, it is known (see \cite[Theorem XIII.82]{ReSi78}) that $D_y^2 -  \langle y \rangle^{-\gamma}$ as infinitely many negative eigenvalues and Lemma \ref{lem1Dbis} follows from \eqref{hyplambdabis}.
\end{proof}

\Bk

\section{Finite number of negative eigenvalue for perturbation by short range potential} \label{S:Finitenumber}

The aim of this section is to prove Theorem \ref{thm3}. In Section \ref{ss42}, using the Birman-Schwinger principle, we reduce the proof to the analysis of some compact operator involving the contribution of the small energies ($\lambda_{m,n}(k) \leq \nu <<1$). Exploiting that the eigenfunctions associated with $\lambda_{m,n}(k)$  are localized near $e^k$, we obtain in Section \ref{ss43} an upper bound of the counting function including interactions between the behavior in $r$ and $z$ via a convolution product and the Fourier transform w.r.t. $z$. Then, exploiting a uniform lower bound of the band functions (see Section \ref{ss41}), we are able to prove Theorem \ref{thm3} by computing the Hilbert-Schmidt norm of a canonical operator and by using standard Young inequality (see Section \ref{ss44}).
\Bk

\subsection{Uniform estimate for the one-dimensional problem} \label{ss41}
\Bk
In order to prove Theorem \ref{thm3} we need an uniform lower bound on the band functions near $0$.
\begin{lem}\label{Lem:Minoration}
Let $\nu_{0}>0$. There exists $C_{0}>0$ such that for all $(m,n,h)$ satisfying $\mu_{m,n}(h)\leq \nu_{0}$ we have 
$$ \mu_{m,n}(h) \geq C_{0}n h$$
\end{lem}
\paragraph{Sketch of the proof}
For convenience, first we work with the operator 
$$\gg_{m}(h)=-h^2\frac{1}{\rho}\partial_{\rho}\rho\partial_{\rho}+V_{h}^{m} \ \ 
\mbox{with}  \ \ 
V_{h}^{m}(\rho):=\log(\rho)^2+h^2\frac{m^2}{\rho^2}\, .$$
We notice that in the sense of quadratic form we have $\gg_{m}(h) \geq \gg_{0}(h)$ and $\dom(\gg_{m}(h))\subset \dom(\gg_{0}(h))$, therefore for all $m\in \Z$ there holds $\mu_{m,n}(h) \geq \mu_{0,n}(h)$ and it is sufficient to prove the result for $m=0$. 

We will split the proof depending on which region belongs the parameter $h$: 
\begin{enumerate}
\item For $h\in (0,h_{0})$ with $h_{0}$ to be chosen, we will use the semi-classical analysis and the Agmon estimates on the eigenfunctions in order to compare $\gg_{0}(h)$ with more standard operators. The idea is to bound from below the potential $\log^2\rho$ on a suitable interval by a quadratic potential such that the associated operator has known spectrum.
\item Since $h\to\mu_{0,n}(h)$ is unbounded for large $h$, there exists $h_{\nu_{0}}$ such that for $h \geq h_{\nu_{0}}$ the eigenvalues $\mu_{m,n}(h)$ are outside the region $\{ \mu \leq \nu_{0} \}$.
\item On the compact $[h_{0},h_{\nu_{0}}]$, since $n\to\mu_{0,n}(h)$ is unbounded for large $n$, we may find $N\geq1$ such that for $n \geq N$ the eigenvalues $\mu_{m,n}(h)$ are outside the region $\{ \mu \leq \nu_{0} \}$. Therefore the Lemma is clear on this region since we have to deal with a finite number of eigenvalues.
\end{enumerate}

\paragraph{proof}

Assume $\mu_{m,n}(h) \leq \nu_{0}$. Denote by $0<\rho_{1}<1<\rho_{2}$ the two real numbers (depending on $\nu_{0}$) such that
 $$
 \log^2(\rho_{1})=\log^2(\rho_{2})=\nu_{0} \, .
 $$
 Set $\rho_{1}'\in (0,\rho_{1})$, $\rho_{2}'\in (\rho_{2},+\infty)$ and $I(\nu_{0}):=(\rho_{1}',\rho_{2}')$. Let $M(\nu_{0}):=\min(\gphi_{0}(\rho_{1}'),\gphi_{0}(\rho_{2}'))$ where $\gphi_{0}$ is defined by 
$$\gphi_{0}(\rho):=\left|\int_{1}^{\rho}\sqrt{\left((\log \rho)^2-\nu_{0}\right)_{+}} \rd \rho\right|$$
By construction we have $M(\nu_{0})>0$ and the Agmon estimate \eqref{E:Agmmeg0} provides $h_{0}>0$ such that (uniformly in $n$):
$$
 \forall h\in (0,h_{0}), \quad \int_{\complement I(\nu_{0})} |\gu_{0,n}(\rho,h)|^2 \rho \rd \rho \leq C(\nu_{0},\beta) e^{-\beta M(\nu_{0})/h} 
$$
where $\beta\in (0,1)$ is set. 

Recall that $\tu_{m,n}(\rho,h)=\sqrt{\rho}\, u_{m,n}(\rho,h)$ is a normalized eigenfunction of $\tgg_m(h)=\cM \gg_{m}(h) \cM^{*}$ associated with the eigenvalue $\mu_{m,n}(h)$. It satisfies
\begin{equation}
\label{E:agmonendehorsI}
\forall h\in (0,h_{0}), \quad \int_{\complement I(\nu_{0})} |\tgu_{0,n}(\rho,h)|^2 \rd \rho \leq C(\nu_{0},\beta) e^{-\beta M(\nu_{0})/h} 
 \end{equation}
\begin{rem}\label{Rq:agmonendehorsI}
Since $\gg_{m}(h) \geq \gg_{0}(h)$, in the sense of quadratic form, the above estimate \eqref{E:agmonendehorsI} holds also for $\tgu_{m,n}$:
$$\forall h\in (0,h_{0}), \quad \int_{\complement I(\nu_{0})} |\tgu_{m,n}(\rho,h)|^2  \rd \rho
= \int_{\complement I(\nu_{0})} |\gu_{m,n}(\rho,h)|^2 \rho \rd \rho \leq C(\nu_{0},\beta) e^{-\beta M(\nu_{0})/h}
$$ 
uniformly with respect to $(m,n)$ such that $\mu_{m,n}(h) \leq \nu_{0}$. This estimate will be used in Section \ref{S:prthm3}.
\end{rem}
Set $\epsilon_{0}\in (0,\rho_{1}')$.
Let $(\chi_{j})_{j=1,2}\in \cC^{\infty}(\R_{+},[0,1])$ be a partition of the unity of $\R_{+}$ such that $\chi_{1}^2+\chi_{2}^2=1$ with $\chi_{2}=0$ on $I(\nu_{0})$ and $\chi_{2}=1$ on $(0,\rho_{1}'-\epsilon_{0})\cup (\rho_{2}'+\epsilon_{0},+\infty)$. We may assume that there exists $C>0$ such that $\sum_{j}|\nabla\chi_{j}|^2 \leq C$.

 The IMS formula provides for any eigenfunction $\tgu_{0,n}(\cdot,h)$: 
\begin{align*}
\tgq_{0}^{h}(\tgu_{0,n}(\cdot,h)) & = \sum_{j=1,2} \tgq_{0}^{h}(\chi_{j}\tgu_{0,n}(\cdot,h))- \sum_{j=1,2}\|(\nabla \chi_{j})\tgu_{0,n}(\cdot,h) \|^2_{L^2(\R_{+})} 
\\
& \geq \tgq_{0}^{h}(\chi_{1}\tgu_{0,n}(\cdot,h))-C\int_{\supp(\chi_{2}')}|\tgu_{0,n}(\rho,h)|^2  \rd \rho
\end{align*}
and therefore using \eqref{E:agmonendehorsI}:
\begin{equation}
\label{E:minorationparIMS}
\tgq_{0}^{h}(\tgu_{0,n}(\cdot,h)) \geq \tgq_{0}^{h}(\chi_{1}\tgu_{0,n}(\cdot,h))-C(\nu_{0},\beta) e^{-\beta M(\nu_{0})/h} .
\end{equation}
We now bound from below $\tgq_{0}^{h}(\chi_{1}\tgu_{0,n}(\cdot,h))$ using a lower bound on the potential. We have 
\begin{equation}
\label{E:minorationpotentiel}
\exists C(\nu_{0})\in (0,1), \forall \rho \in J(\nu_{0}), \quad C(\nu_{0}) (\rho-1)^2 \leq \log^2\rho 
\end{equation}
where we have denoted $J(\nu_{0}):=(\rho_{1}'-\epsilon_{0},\rho_{2}'+\epsilon_{0})$. 

Assume $n\neq n'$. Since $\langle \tgu_{0,n}(\cdot,h), \tgu_{0,n'}(\cdot,h)  \rangle_{L^2(\R_{+})}=0$,  we deduce from \eqref{E:agmonendehorsI} that 
\begin{equation}
\label{E:quasiorthogonalitetronc}
\left| \langle \chi_{1}\tgu_{0,n}(\cdot,h), \chi_{1}\tgu_{0,n'}(\cdot,h)  \rangle_{L^2(\R_{+})} \right|  \leq C(\nu_{0},\beta) e^{-\beta M(\nu_{0})/h}.
\end{equation}

Let us introduce the harmonic oscillator
$$\gg^{\rm low}(h):=-h^2 \partial^2_{\rho}+(\rho-1)^2, \quad \rho \in \R$$
initially defined on $\cC^{\infty}_{0}(\R)$ and close on  $L^2(\R)$, whose eigenvalues are $\{(2n-1)h\}_{n \in \N^*}$.
Due to \eqref{E:minorationpotentiel} and since $\supp(\chi_{1})=J(\nu_{0})$ we have 
%\begin{equation}
\begin{multline}
\label{E:minorationoperator}
\tgq_{0}^{h}(\chi_{1}\tgu_{0,n}(\cdot,h)) \geq  C(\nu_{0}) \langle \gg^{\rm low}(h)\chi_{1} \tgu_{0,n}(\cdot,h),\chi_{1}\tgu_{0,n}(\cdot,h)\rangle_{L^2(\R)}
\\
 - \frac{h^2}{4(\rho_{1}'-  \epsilon_{0})^2}
\|\chi_{1}\tgu_{0,n}(\cdot,h)\|^2_{L^2(\R_{+})}
 %\langle \chi_{1} \tgu_{0,n}(\cdot,h),\chi_{1}\tgu_{0,n}(\cdot,h)\rangle_{L^2(\R)}
\end{multline}
%\end{equation}
where in the right hand side, $\chi_1\tgu_{0,n}$, extended by $0$ on $\R_-$, is also considered as a function defined on $\R$. 

 Recall \eqref{E:quasiorthogonalitetronc}, the min-max principle combined with \eqref{E:minorationoperator} provides 
\begin{equation}
\label{E:minorationpartiechi1}
\tgq_{0}^{h}(\chi_{1}\tgu_{0,n}(\cdot,h)) \geq \Big(C(\nu_{0})(2n-1)h-\widetilde{C}(\nu_{0},\beta) h^2\Big)\|\chi_{1}\tgu_{0,n}(\cdot,h)\|^2_{L^2(\R_{+})}. 
\end{equation}

Using \eqref{E:agmonendehorsI} we get $|1-\| \chi_{1}\tgu_{0,n}(\cdot,h)\|^2_{L^2(\R_{+})}| \leq C(\nu_{0},\beta) e^{-\beta M(\nu_{0})/h}$.

\Bk

Therefore combining \eqref{E:minorationparIMS} and \eqref{E:minorationpartiechi1} we have proved the existence of $h_{0}>0$ and $C_{0}>0$ such that for all $(0,n,h)$ such that $\mu_{0,n}(h) \leq \nu_{0}$ we have
$$ \forall h \in (0,h_{0}), \quad \mu_{0,n}(h)= \tgq_{0}^{h}(\tgu_{0,n}(\cdot,h)) \geq C_{0} nh. $$
We now have to deal with the region $h\in (h_{0},+\infty)$. Since $\mu_{0,n}(h)$ tends to $+\infty$ as $h$ tends to $+\infty$, there exists $h_{\nu_{0}}>0$ such that
$$\forall n \in \N^*, \forall h\geq h_{\nu_{0}}, \quad \mu_{0,n}(h) \geq \nu_{0} \, .$$
Therefore we are led to prove the lower bound for $h\in [h_{0},h_{\nu_{0}}]$. Since for all $h>0$ the sequence $(\mu_{m,n}(h))_{n \geq1}$ converges toward $+\infty$, there exists $n(h)$ such that for all $n \geq n(h)$ we have $\mu_{m,n}(h) \geq \nu_{0}$. Due to a compact 	argument we find $N\in \N^*$ such that 
$$\forall n > N, \forall h\in [h_{0},h_{\nu_{0}}], \quad \mu_{0,n}(h) \geq \nu_{0} \, .$$
Define $C(h):=\min_{1 \leq n \leq N} \mu_{0,n}(h)/n$ and $C:=\min_{h\in [h_{0},h_{\nu_{0}}]}\frac{C(h)}{h}$. We clearly have $C>0$ and by construction, for all $(n,h)\in \N^* \times [h_{0},h_{\nu_{0}}]$ such that $\mu_{0,n}(h) \leq \nu_{0}$ we have
$$\mu_{0,n}(h) \geq Cnh$$
 therefore the lemma is proved for $h\in [h_{0},h_{\nu_{0}}]$.
 
\begin{rem}
In \eqref{E:minorationpartiechi1}, the remainder term of order $h^2$ involves the contributions of $\frac{-h^2}{4\rho^2}$ and has been controlled on $J(\nu_{0})$. Another strategy, which improves the remainder term, \Bk would have been to work in the weighted space $L^2_\rho(\R_+)$ and to consider 
$$\gg^{\rm low}(h):=-h^2 \frac{1}{\rho}\partial_{\rho}\rho\partial_{\rho}+(\rho-1)^2, \quad \rho>0.$$
In this case, \eqref{E:minorationpartiechi1} is replaced by
$$\gq_{0}^{h}(\chi_{1}\gu_{0,n}(\cdot,h)) \geq C(\nu_0)(\zeta_{n}(h)-C(\nu_{0},\beta) e^{-\beta M(\nu_{0})/h})\|\chi_{1}\gu_{0,n}(\cdot,h)\|^2_{L^2_{\rho}(\R_{+})} $$
with $\zeta_{n}(h)$ the $n$-th eigenvalue of the operator 
$\gg^{\rm low}(h)$. These eigenvalues have already been studied in \cite[Section 4.2]{Yaf08} and \cite{Pof13-II} and they can be bounded from below by $C_1nh$ by exploiting the results from \cite{Pof13-II}.
\Bk 

\end{rem}

\subsection{Bring the norm of a canonical operator}\label{S:prthm3}\label{ss42}

%\begin{proof} 
Let $\lambda>0$, for simplicity we denote by $\mathcal{N}(\lambda):=\mathcal{N}_{\bA,V}(\lambda)$ the number of negative eigenvalues of $H_{\bA}-V$ below $-\lambda$: \Bk
$$\mathcal{N}(\lambda):= \sharp \Big( \spec(H_{\bA}-V) \cap ]-\infty,-\lambda]\Big).$$ 
We want to prove that there exists $C>0$ independent of $\lambda$, such that $\mathcal{N}(\lambda) \leq C$.
Let us introduce the axisymmetric non negative potential 
\bel{D:V0}
V_0(r,z):=  \langle r \rangle^{-\alpha} \, v_\perp(z).
\ee
 The assumption \eqref{hypV} means that $V\leq V_0$. Then the min-max principle gives:
 \bel{eq2}
 \mathcal{N}(\lambda)\leq \mathcal{N}_0(\lambda):= \sharp \Big( \spec(H_{\bA}-V_0) \cap ]-\infty,-\lambda]\Big).
 \ee

According to the Birman-Schwinger principle, for $\lambda>0$, 
\bel{eq3}
\mathcal{N}_0(\lambda) = n_+\Big(1, V_0^{\frac12} (H_{\bA}+\lambda)^{-1} V_0^{\frac12}\Big), 
\ee
where for a self-adjoint operator $T$, $n_+(s,T): = {\rm Tr}\,\one_{(s,\infty)}(T);$ is the counting function of positive eigenvalues of $T$. 

Fix a real number $\nu>0$ (chosen sufficiently small later) and let us introduce the orthogonal projections
$P_\nu :=  {\bf 1}_{[0,\nu]} (H_{\bA})$
and $\overline{P_\nu}:= I - P_\nu= {\bf 1}_{]\nu, + \infty[}(H_{\bA})$. 

Since $H_{\bA} \overline{P_\nu} \geq \nu$, the compact operator $V_0^{\frac12} (H_{\bA}+\lambda)^{-1}  \overline{P_\nu} V_0^{\frac12} $ is uniformly bounded with respect to $\lambda \geq 0$ and from the Weyl inequality, for any $\epsilon >0$, we have:
\bel{eq5}
n_+\Big(1, V_0^{\frac12} (H_{\bA}+\lambda)^{-1} V_0^{\frac12}\Big) \leq n_+\Big(1-\epsilon , V_0^{\frac12} (H_{\bA}+\lambda)^{-1} P_\nu V_0^{\frac12}\Big) + C_\nu, \quad C_\nu \geq 0.
\ee
According to the decomposition:
$$H_{\bA}= \Phi^{*} \cF_{3}^{*}  \left(  \sum^{\bigoplus}_{(m,n) \in \Z\times \N^*} \int_{k\in \R}^{\bigoplus}  \lambda_{m,n}(k) P_{m,n}(k)  \rd k \right) \cF_{3}  \Phi,$$
with $P_{m,n}(k):f\mapsto \langle f, u_{m,n}(\cdot,k)\rangle u_{m,n}(\cdot,k)$,  the orthogonal projection onto $u_{m,n}(.,k) \in L^2(\R_+,  r\rd r)$, we have 
$$V_0^{\frac12} (H_{\bA}+\lambda)^{-1} P_\nu V_0^{\frac12} = V_0^{\frac12}   
\Phi^{*} \cF_{3}^{*}  \left(  \sum^{\bigoplus}_{(m,n) \in \Z\times \N^*}\int_{k\in \R}^{\bigoplus}  P_{m,n}(k)   \frac{{\bf 1}_{[0,\nu]}(\lambda_{m,n}(k)) }{\lambda_{m,n}(k)+\lambda} \rd k \right) \cF_{3}  \Phi V_0^{\frac12} .$$

Since $V_0$ is axisymmetric, this operator is unitarily equivalent to the direct sum of 
$$K_{\nu,m}(\lambda):=V_0^{\frac12}\cF_{3}^{*}  \left(  \int_{k\in \R}^{\bigoplus}  \sum^{\bigoplus}_{n \in \N^*} \tP_{m,n}(k)   \frac{{\bf 1}_{[0,\nu]}(\lambda_{m,n}(k)) }{\lambda_{m,n}(k)+\lambda} \rd k \right) \cF_{3} V_0^{\frac12},$$
defined in 
$L^2(\R_+\times \R, \rd r \rd z)$, with $\tP_{m,n}(k):= \cM^* P_{m,n}(k) \cM = \langle ., \tu_{m,n}(k)\rangle \tu_{m,n}(k,.)$,  the orthogonal projection onto $\tu_{m,n}(.,k) \in L^2(\R_+,  \rd r)$, $\tu_{m,n}(r,k)=\sqrt{r} u_{m,n}(r,k)$.

Let us prove that for some $s \in ]0,1[$, there exists $\nu$ sufficiently small such  for any $m \in \Z$ and any $\lambda>0$
\bel{eq7}
n_+(s, K_{\nu,m}(\lambda))=0.
\ee

Then Theorem \ref{thm3}, is a consequence of  \eqref{eq2}, \eqref{eq3}, \eqref{eq5} and \eqref{eq7}.

Let us introduce the operator:
 $$S_m(\lambda)\; : \; L^2(\R, l^2(\N^*)) \longrightarrow L^2(\R_+\times \R, \rd r \rd z),$$
 defined, for $(g_n(.))_{n\in \N^*} \in  L^2(\R, l^2(\N^*))$ by
\bel{D:Sm}
 S_m(\lambda)(g_n)(r, z):= \frac{V_0^{\frac12}(r,z)}{\sqrt{2\pi}} \sum_{n \in \N^*} \int_{\R} g_n(k) \frac{e^{izk}{\bf 1}_{[0,\nu]}(\lambda_{m,n}(k)) }{(\lambda_{m,n}(k)+\lambda)^\frac12} {\tu_{m,n}(r,k)}\rd k,
 \ee
and its adjoint defined for $f \in L^2(\R_+\times \R, \rd r\rd z)$, by
$$S_m(\lambda)^*(f)_n(k)= \frac{1}{\sqrt{2\pi}} \frac{{\bf 1}_{[0,\nu]}(\lambda_{m,n}(k))}{(\lambda_{m,n}(k)+\lambda)^\frac12} \int_{\R_+\times \R}e^{-izk} \overline{\tu_{m,n}(r,k)}(V_0^{\frac12} f)(r,z)\rd r \rd z.$$
We have:
$$ K_{\nu,m}(\lambda)= S_m(\lambda)\, S_m(\lambda)^*,$$
and since 
\bel{eq8}
n_+(s, K_{\nu,m}(\lambda))= n_+(s,   S_m(\lambda)\, S_m(\lambda)^*    )= n_+(s,   S_m(\lambda)^*\, S_m(\lambda)   ),
\ee
 we have to prove that for $\nu$ sufficiently small, the $L^2-$norm of $S_m(\lambda)^*\, S_m(\lambda)$ admits an upper bound by $s<1$ uniformly with respect to $m \in \Z$ and $\lambda>0$. 
 
 \subsection{Computations on the integral kernel of the canonical operator}\label{ss43}
 \begin{prop}
 \label{P:controlenormeHS}
Let $V_{0}$ defined in \eqref{D:V0} and $S_{m}(\lambda)$ defined in \eqref{D:Sm}. Then 
there exist $C>0$ and $\nu_{0}>0$ such that for all $\nu\in (0,\nu_{0})$, the following upper bound of the Hilbert-Schmidt norm holds:
\begin{equation}
\label{M:symetriqueiota}
\|S_{m}(\lambda)^{*}S_{m}(\lambda)\|_2^2 \leq
C\sum_{n,n'}\int_{k}\int_{k'} \iota_{m,n'}(k',\nu) \iota_{m,n}(k,\nu)|\widehat{v_{\perp}}(k'-k) |^2\rd k' \rd k
\end{equation}
 where we have set 
$$\iota_{m,n}(k,\nu):=\frac{{\bf 1}_{[0,\nu]}(\lambda_{m,n}(k))}{\lambda_{m,n}(k)+\lambda} e^{-\alpha k} \, .$$

 \end{prop}
 \begin{proof}
  We check that $S_{m}(\lambda)^{*}S_{m}(\lambda): L^2(\R,l^2(\N^*))\longrightarrow L^2(\R,l^2(\N^*))$ corresponds with
 \begin{multline}
 \left(S_{m}(\lambda)^{*}S_{m}(\lambda)(g_{n'})\right)_{n} (k) =
 \\
\frac{1}{2\pi}  L_{m,n}(k)\int_{z}\int_{r}\overline{\tu_{m,n}(r,k)}V_{0}(r,z)\sum_{n'}\int_{k'}g_{n'}(k')L_{m,n'}(k')\tu_{m,n'}(r,k')e^{iz(k'-k)}\rd k' \rd r \rd z
 \end{multline}
 where we have denoted 
$$L_{m,n}(k):=\frac{{\bf1}_{[0,\nu]}(\lambda_{m,n}(k))}{\sqrt{\lambda_{m,n}(k)+\lambda}} \, .$$
The integral kernel of this operator is 
\begin{align*}
\gN_{m,n,n'}(k,k')&:=
%\frac{{\bf 1}_{[0,\nu]}(\lambda_{m,n}(k)){\bf 1}_{[0,\nu]}(\lambda_{m,n'}(k'))}{\sqrt{\lambda_{m,n}(k)+\lambda}\sqrt{\lambda_{m,n'}(k')+\lambda}}
L_{m,n}(k)L_{m,n'}(k')\int_{r}\int_{z}V_{0}(r,z)\overline{\tu_{m,n}(r,k)}\tu_{m,n'}(r,k')e^{iz(k-k')}\rd z \rd r \, .
\\
&=L_{m,n}(k)L_{m,n'}(k')
%\frac{{\bf 1}_{[0,\nu]}(\lambda_{m,n'}(k'))}{\sqrt{\lambda_{m,n'}(k')+\lambda}}\frac{{\bf 1}_{[0,\nu]}(\lambda_{m,n}(k))}{\sqrt{\lambda_{m,n}(k)+\lambda}}
\widehat{v_{\perp}}(k'-k)
\int_{r}\langle r \rangle^{-\alpha}\overline{\tu_{m,n}(r,k)}\tu_{m,n'}(r,k')\rd r.
\end{align*}

%At first one may compute the trace norm of $S_{m}(\lambda)^{*}S_{m}(\lambda)$ by looking at the trace of the kernel. As we will se the Hilbert-Schmidt norm is more suitable since it takes into account the interaction between the 
Then the Hilbert-Schmidt norm is given by 
\begin{multline}
\label{E:normHSexplicit}
4\pi^2 \|S_{m}(\lambda)^{*}S_{m}(\lambda)\|_2^2=
\\
%\frac{1}{4\pi^2} 
\sum_{n,n'} \int_k \int_{k'} \int_r 
L_{m,n}(k)^2L_{m,n'}(k')^2
%\frac{{\bf 1}_{[0,\nu]}(\lambda_{m,n'}(k'))}{\lambda_{m,n'}(k')+\lambda}\frac{{\bf 1}_{[0,\nu]}(\lambda_{m,n}(k))}{\lambda_{m,n}(k)+\lambda}| 
|\widehat{v_{\perp}}(k'-k)|^2 \left| \int_{r}\langle r \rangle^{-\alpha}\overline{\tu_{m,n}(r,k)}\tu_{m,n'}(r,k')\rd r\right|^2 \rd k \rd k'.
\end{multline}
Set $\nu_{0}>0$ and $(m,n,k)$ such that $\lambda_{m,n}(k) \leq \nu_0$. Applying Remark \ref{Rq:agmonendehorsI} 
%of Section \ref{S:prthm3}
 we know that there exists $I_k(\nu_{0}):=[\rho'_1 e^k, \rho'_2 e^k]$, $\rho'_1 < 1 < \rho'_2$, such that for any $k \geq k_0$ sufficiently large (independent of $(m,n)$),
$$\int_{\complement I_{k}(\nu_{0})} \!\!\!\langle r \rangle^{-\alpha}
\mid \tu_{m,n}(k,r) \mid^2dr  \leq \int_{\complement I_{k}(\nu_0)} \!\!\!
\mid \tu_{m,n}(k,r) \mid^2dr\leq C(\nu_{0},\beta) e^{-\beta M(\nu_0)e^k}
$$ 
with $\beta \in (0,1)$ and  $M(\nu_{0})>0$. 
On the other hand, on $I_k(\nu_{0})$, we have
$$\int_{ I_{k}(\nu_{0})} \!\!\!\langle r \rangle^{-\alpha}
\mid \tu_{m,n}(k,r) \mid^2dr  \leq C(\nu_{0})e^{-\alpha k} \int_{ I_{k}(\nu_{0})} \!\!\!
\mid \tu_{m,n}(k,r) \mid^2dr\leq  C(\nu_{0})e^{-\alpha k}.
$$ 
Consequently, 
\Bk 
\bel{Esti1}
\int_{\R_+} \!\!\langle r \rangle^{-\alpha}
\mid \tu_{m,n}(k,r) \mid^2dr = O(e^{-\alpha k}),
\ee
uniformly with respect to $(m,n,k) \in \Z \times\N^*\times\R$ satisfying $\lambda_{m,n}(k)\leq \nu_{0}$. Using the Cauchy-Schwarz inequality we deduce from \eqref{E:normHSexplicit} that  for all $\nu\in (0,\nu_{0})$:
\begin{equation}
\label{M:symetriqueiotabis}
\|S_{m}(\lambda)^{*}S_{m}(\lambda)\|_2^2 \leq
C\sum_{n,n'}\int_{k}\int_{k'} \iota_{m,n'}(k',\nu) \iota_{m,n}(k,\nu)|\widehat{v_{\perp}}(k'-k)|^2 \rd k' \rd k
\end{equation}
and the lemma is proved
\end{proof}
We notice that the influence of $V$ appears as an interaction between the behaviors in $r$ and $z$ via a convolution product in the phase space. We now estimate the norm of the function $\iota_{m,n}(k,\nu)$:
\begin{lem}
\label{L:normedeiota}
There exists $C>0$ and $\nu_{0}>0$ such that for all $(m,n,k)\in \Z\times \N^*\times \R$, we have
$$\forall \nu\in (0,\nu_{0}), \forall q\geq1, \quad \|\iota_{m,n}(\cdot,\nu)\|_{L^q} \leq C \frac{\nu^{\alpha-1}}{n^{\alpha}} \, .$$
\end{lem}
\begin{proof}
Set $\nu_{0}>0$ and assume $\lambda_{m,n}(k) \leq \nu_{0}$.  According  to Lemma \ref{Lem:Minoration} there exists $C_0>0$ such that 
\bel{eq13}
\lambda_{m,n}(k) \geq C_0 n e^{-k},
\ee 
uniformly with respect to $(m,n,k) \in \Z\times\N^*\times\R$. Then for $\nu\in (0,\nu_{0})$ there holds 
${\bf 1}_{[0,\nu]}(\lambda_{m,n}(k))\leq {\bf 1}_{[0,\frac{\nu}{C_0}]}(n e^{-k})$ and  for any $\lambda>0$ we have
\begin{align*}
\| \iota_{m,n} \|_{L^q}^q =\int_{k} \frac{{\bf 1}_{[0,\nu]}(\lambda_{m,n}(k))}{(\lambda_{m,n}(k)+\lambda)^q} e^{-\alpha q k } \rd k
&
 \leq \int_{k \geq \log\frac{C_{0}n}{\nu}}\frac{1}{(\lambda_{m,n}(k)+\lambda)^q} e^{-\alpha q k } \rd k 
\\
& \leq \frac{1}{(C_{0}n)^q}\int_{k \geq \log\frac{C_{0}n}{\nu}} e^{(-\alpha+1) q k } \rd k
\\
& =\frac{1}{q(\alpha-1)(C_{0}n)^q} \left(\frac{\nu}{C_{0}n}\right)^{(\alpha-1)q}
\end{align*}
and the lemma is proved.
\end{proof}
\subsection{Convergence of the series and proof of Theorem \ref{thm3}}\label{ss44}
We notice that the r.h.s of \eqref{M:symetriqueiota} coincides with
$$ \sum_{n,n'}\int_{k} \iota_{m,n}(k,\nu) (\iota_{m,n'}(\cdot,\nu)\ast | \widehat{v_{\perp}}|^2) (k) \rd k \, . $$
Assume that $v_{\perp}\in L^p$ with $p \in [1,2]$. Then $| \widehat{v_{\perp}}|^2 \in L^{p'/2}$ with $p'=\frac{p}{p-1}\geq 2$. %We ask $\frac{p'}{2}\geq1$ that means $p\leq2$ and  
Young's inequality provides for all $q\geq1$:
$$\|\iota_{m,n'}\ast | \widehat{v_{\perp}}|^2 \|_{L^r} \leq  \|\iota_{m,n'} \|_{L^q} \| v_{\perp} \|_{L^p}$$ 
where $\frac{2}{p'}+\frac{1}{q}=1+\frac{1}{r}$. We now use Holder's inequality combined with lemma \ref{L:normedeiota} and we get for all $(m,n,n')$:
$$\forall \nu \in (0,\nu_{0}), \quad \int_{k} \iota_{m,n}(k,\nu) (\iota_{m,n'}(\cdot,\nu)\ast | \widehat{v_{\perp}}|^2) (k) \rd k \leq C  \| v_{\perp} \|_{L^p}\frac{\nu^{2\alpha-2}}{n^{\alpha}n'^{\alpha}}$$
Since $\alpha>1$, we get 
$$\sum_{n,n'}\int_{k} \iota_{m,n}(k,\nu) (\iota_{m,n'}(\cdot,\nu)\ast | \widehat{v_{\perp}}|^2) (k) \rd k=O(\nu^{2\alpha-2})  \| v_{\perp} \|_{L^p}\sum_{n\geq1}\frac{1}{n^{\alpha}}\sum_{n'\geq1}\frac{1}{(n')^{\alpha}}$$
and therefore using Proposition \ref{P:controlenormeHS}:
\bel{}
\|S_{m}(\lambda)^{*}S_{m}(\lambda)\|_2^2 =O(\nu^{2\alpha-2})
\ee
which, for $\alpha>1$, tends to $0$ with $\nu$,  uniformly with respect to $(m,\lambda) \in \Z\times(0,+ \infty)$. Then, \eqref{eq7} follows from \eqref{eq8}. In conclusion the hypotheses we have used on $V(r,z)$ is $V(r,z) \leq \langle r \rangle^{-\alpha}v_{\perp}(z)$ with $\alpha>1$ and $v_{\perp}\in L^p(\R)$, $p\in [1,2]$ and we deduce Theorem \ref{thm3}.

\bibliographystyle{mnachrn}

%\begin{thebibliography} {[10]}
%\frenchspacing \baselineskip=12 pt plus 1pt minus 1pt
%\bibitem{AHS} {\sc J. Avron, I. Herbst, B. Simon},
%{\em Schr\"{o}dinger operators with  magnetic  fields.  I. General
%interactions}, Duke Math. J. {\bf 45} (1978), 847--883.
%
%\bibitem{rs4} {\sc M. Reed, B. Simon}, {\em Methods  of
%Modern  Mathematical Physics. IV. Analysis of Operators}, Academic
%Press, New York, 1979.
%
%\end{thebibliography}

\end{document}